\documentclass[12pt,twoside]{article}
\usepackage{amsmath,amsbsy,amsfonts}
\usepackage[latin1]{inputenc}
\usepackage{amsmath,amsbsy,amsfonts}
\usepackage{color}
\usepackage{tikz}
\newtheorem{theorem}{Theorem}[section]

\newtheorem{definition}{Definition}[section]
\newtheorem{lemma}{Lemma}[section]
\newtheorem{proposition}{Proposition}[section]
\newtheorem{remark}{Remark}[section]

\newtheorem{corollary}{Corollary}[section]
\newcommand{\fim}{\hfill\rule{2mm}{2mm}}
\def\proof{\mbox {\it Proof~}}
\def\nd{\noindent}

\def\r{\mathbb{R}}
\begin{document}
\title{
\vspace{-1.5in} 
\textbf{Equivalent conditions for existence of three solutions for a problem with discontinuous and strongly-singular  terms}}

\author{
{\bf\large Carlos Alberto Santos}\footnote{Carlos Alberto Santos acknowledges
the support of CAPES/Brazil Proc.  $N^o$ $2788/2015-02$.} \\
{\it\small Universidade de Bras\'ilia, Departamento de Matem\'atica}\\
{\it\small   70910-900, Bras\'ilia - DF - Brazil}\\
{\it\small e-mail: csantos@unb.br}\vspace{1mm}\\
{\bf\large Lais Santos}\\
{\it\small Universidade Federal de Vi\c cosa, Departamento de Matem\'atica}\\
{\it\small   36570.000, Vi\c cosa - MG - Brazil}\\
{\it\small e-mails: matmslais@gmail.com}\\\vspace{1mm}
{\bf\large Marcos L. M. Carvalho}\\
{\it\small Universidade Federal de Goi\'as, Instituto de Matem\'atica}\\
{\it\small   74690-900, Goi\^ania - GO - Brazil}\\
{\it\small e-mails: marcos$\_$leandro$\_$carvalho@ufg.br}
}

\date{}
\maketitle
 \vspace{-1,0cm}

\mbox{}

\begin{abstract}
\noindent	In this paper, we are concerned with a Kirchhoff problem in the presence of a strongly-singular term perturbed by a discontinuous nonlinearity of the Heaviside type in the setting of Orlicz-Sobolev space. The presence of both  strongly-singular and non-continuous terms bring up difficulties in associating a differentiable functional  to the problem with finite energy in the whole space $W_0^{1,\Phi}(\Omega)$. To overcome this obstacle, we established an optimal condition for the existence of $W_0^{1,\Phi}(\Omega)$-solutions to a strongly-singular problem, which allows us to constrain the energy functional to a subset of $W_0^{1,\Phi}(\Omega)$ to apply techniques of convex analysis and generalized gradient in Clarke sense.
\end{abstract}

\nd {\it \footnotesize 2010 Mathematics Subject Classifications:} {\scriptsize 35J25, 35J62, 35J75, 35J20, 35D30, 35B38}\\
\nd {\it \footnotesize Key words}: {\scriptsize Non-local Kirchhoff problems, Strongly-singular nonlinearity,  Discontinuous perturbation,  $\Phi$-Laplacian operator. }

\section{Introduction}
In this paper, we are concerned in presenting equivalent conditions for the  existence of three  solutions  for the quasilinear problem 
$$
(Q_{\lambda, \mu})~~\left\{
\begin{array}{l}
-M\left(\displaystyle\int \Phi(|\nabla u|)dx\right)\Delta_\Phi u =\mu b(x)u^{-\delta} + {\lambda }f(x,u) ~ \mbox{in } \Omega,\\
u>0 ~ \mbox{in }\Omega,~~
u=0  ~ \mbox{on }\partial\Omega,
\end{array}
\right.
$$
which are linked to  an optimal compatibility condition between $(b,\delta)$ for  existence of solution to 
the strongly-singular problem 
$$
(S)~~\left\{
\begin{array}{l}
-\Delta_\Phi u =  b(x)u^{-\delta}  ~ \mbox{in } \Omega,\\
u>0 ~ \mbox{in }\Omega,~~
u=0  ~ \mbox{on }\partial\Omega
\end{array}
\right.
$$
 with the boundary condition still in the sense of the trace.
 
Here, $M:[0,\infty) \to [0,\infty) $ is a continuous function,  $ f:\Omega\ \times (0,\infty) \to (0,\infty) $ is of Heaviside type,  $0<b \in L^1(\Omega)$, $\delta>1$, $\lambda, \mu > 0$ are real parameters. Moreover, $-\Delta_\Phi u = -\mbox{div}(a(|\nabla u|)\nabla u)$ stands for the $\Phi$-Laplacian operator, where $a:(0,\infty) \to (0,\infty)$ is a $C^1$-function that defines the increasing  homeomorphism $\phi : \mathbb{R} \to \mathbb{R}$ given by $$\phi(t) = \left\{\begin{array}{l}
a(|t|)t ~~ \mbox{if} ~t \neq 0, \\
0 ~~\mbox{if} ~t = 0,
\end{array}\right.$$ 
whose the associated  N-function  $\Phi : \mathbb{R} \to \mathbb{R}$ is given by $\Phi(t) = \int_0^{|t|}\phi(s)ds$.

The issue about existence of three solutions for a suitable range of parameters $\lambda,\mu>0$, for particular forms of Probem $(Q_{\lambda,\mu})$, has been considered in the literature recently, principally in the context of non-singular problems ($\delta<0$) and in the case in which $f$ is continuous, see for instance \cite{MR2764385},  \cite{MR3646986}, \cite{MR3126590}, \cite{MR2601787}, \cite{MR2819316}, \cite{MR3274355} and references therein.  There are few works for  singular nonlinearities, we quote for example \cite{MR3530211}, \cite{faraci} and \cite{MR2819316} who considered $\Phi(t) = |t|^p/p $, $t>0$,  $1<p <\infty$ and $M\equiv1$ in $(Q_{\lambda,\mu})$.

In \cite{MR2819316}, a singular problem for low dimensions was studied, while in \cite{MR3530211} and \cite{faraci} a singular problem for high dimensions was treated, but in both cases $f$ has been considered a Carath\'eodory function with suitable assumptions. More specifically, in \cite{faraci}, the singular perturbation was considered in the weak sense ($0<\delta <1$), while in \cite{MR3530211} they permitted $\delta>1$ by balancing the size of this $\delta$ with the existence of a $0<u \in C_0^1(\overline{\Omega})$ such that the product $b u^{-\delta}$ in $L^{\left(p^*\right)'}(\Omega)$. 

In this paper, we establish an optimal condition to the relationship between the power $\delta>1$ and the potential $b(x)>0$ to existence of three solutions to
 the singular problem $(Q_{\lambda, \mu})$, independent of the dimension $N$,  in the presence of both a discontinuous nonlinearity of the Heaviside type and a non-local term. More precisely, we prove how the existence of three solutions to $(Q_{\lambda, \mu})$ is associated to the existence of solutions still in $W_0^{1,\Phi}(\Omega)$ to the problem $(S)$. Our approach is based on the existence of positive solution to the problem $(S)$, which provides a non-empty effective domain for the energy functional associated to $(Q_{\lambda, \mu})$ and  enable us to apply techniques of the  generalized gradient in Clarke sense to get a multiplicity result.

Besides this, we prove qualitative results about these three solutions. We highlight how the non-local term $M$ should be  to the discontinuity of the function $f$ be effectively attained by the solutions and how the level set of these solutions behaves exactly at the discontinuity point of $f$. To our knowledge, both the results of equivalent conditions and qualitative information on solutions are new in literature.

As our main results will be obtained via variational methods, we need to introduce the energy functional associated to Problem $(Q_{\lambda, \mu})$. To do this, let us denote by $W_0^{1,\Phi}(\Omega)$ the Orlicz-Sobolev space associated to $\Phi$ and extend the function $f$ to $\mathbb{R}$ as $f(x,t) = 0 $ $a.e$ in $\Omega$ and for all $t \leq 0$. From these, the  functional naturally associated to $(Q_{\lambda,\mu})$ is $I : W_0^{1,\Phi}(\Omega) \to \mathbb{R}$ defined by
	\begin{equation*}\label{min}
	I(u) = \hat{M}\left(\displaystyle\int_\Omega \Phi(|\nabla u|)dx \right) - \lambda \displaystyle\int_\Omega F(x,u)dx + \mu\displaystyle\int_\Omega G(x,u)dx,
	\end{equation*}
	where 
	$$\hat{M}(t) = \displaystyle\int_0^{t}M(s)ds,~~F(x,t):=\int_0^tf(x,s)ds $$ and $G: \Omega \times \mathbb{R} \to (-\infty, \infty]$ is defined by
	$$G(x,t) = \left\{\begin{array}{l}
	\frac{-b(x)t^{1-\delta}}{1-\delta} ~~\mbox{for} ~x \in \Omega ~\mbox{and} ~t > 0,\\
	+\infty ~~\mbox{for} ~x \in \Omega ~ \mbox{and} ~ t \leq 0.
	\end{array}\right.$$
	
To ease our future references, let us  rewrite $I$ as  $I=\Psi_1 + \mu\Psi_2$, where
\begin{equation}\label{t1}
\Psi_1(u) =  \hat{M}\left(\displaystyle\int \Phi(|\nabla u|)dx\right) - \lambda \displaystyle\int_\Omega F(x,u)dx
\end{equation}
and 
\begin{equation}\label{t22}
\Psi_2(u) = \displaystyle\int_\Omega G(x,u)dx.
\end{equation}

The main difficulty in treating strongly-singular problems consists in the fact that the energy functional  associated to the equation neither belongs to  $C^1$, in the sense of Fr\'echet differentiability, nor is defined in the whole space $W_0^{1,\Phi}(\Omega)$. In fact, when $ \delta > 1 $ the functional $\Psi_2$ may not be proper, i.e. it may occur $\Psi_2(u)=\infty$, for all $u \in W_0^{1,\Phi}(\Omega)$. 

Another difficulty exploited in this work is the presence of a more general quasilinear operator, which may be even nonhomogeneous. To deal with this situation, we approach the problem $(Q_{\lambda,\mu})$ in Orlicz-Sobolev space setting. Below, let us state the assumptions about $\Phi$ that we will assume throughout this paper.
\begin{itemize}
	\item[($\phi_0$):] $a \in C^1((0, \infty), (0, \infty))$  and $\phi $ is an increasing odd homeomorphisms from $\mathbb{R} $ onto $\mathbb{R}$; 
%	\item[$(\phi_1)$:] \textcolor{red}{$1 < \phi_- := \displaystyle\inf_{t > 0}\frac{t\phi(t)}{\Phi(t)} \leq \displaystyle\sup_{t > 0}\frac{t\phi(t)}{\Phi(t)} := \phi_+ < \infty; $ Retirar essa hipótese. Colocar como observação}
	\item[$(\phi_1)$:] $0 < a_- := \displaystyle\inf_{t > 0}\frac{t\phi'(t)}{\phi(t)} \leq \displaystyle\sup_{t > 0}\frac{t\phi'(t)}{\phi(t)} := a_+ < \infty. $	
\end{itemize}

Let us  denote by $\Phi_*$ the function whose inverse is given by $(\Phi_*)^{-1}(t) = \int_0^{t} \Phi^{-1}(s)s^{-1-1/N}ds $, $t >0$. In order to $\Phi_*$  be a N-function, we need to require  \begin{equation*}\label{imersao}
\int_0^{1} \Phi^{-1}(s)s^{-1-1/N}ds < \infty  ~~\mbox{and} ~~  \displaystyle\int_1^{\infty} \Phi^{-1}(s)s^{-1-1/N}ds = \infty.
\end{equation*} 

In this case,  $\Phi_*$ is a N-function given by  $\Phi_*(t) = \int_0^{|t|}\phi_*(s)ds$ for some  increasing odd homeomorphisms $\phi_*: \mathbb{R} \to \mathbb{R}$. About  $\Phi_*$, we will consider 
\begin{itemize}
	\item[($\phi_2$):] $\phi_+ < \phi_-^* :=  \displaystyle\inf_{t > 0}\frac{t\phi_*(t)}{\Phi_*(t)}$, where 
$1 < \phi_- :=a_- +1\leq  a_+ +1 :=\phi_+.$
\end{itemize}

As another consequence of $(\phi_0)$ and $(\phi_1)$, the Orlicz space $L^{\Phi}(\Omega)$ coincides with the set (equivalence classes) of measurable functions $u : \Omega \to \mathbb{R}$ such that
$\int_\Omega \Phi(|u|)dx < \infty$ and it is a Banach space endowed with the Luxemburg norm 
\begin{equation*}\label{norma1}
\|u\|_\Phi := \inf\left\{\alpha > 0 ~: ~ \displaystyle\int_\Omega \Phi\left(\frac{|u(x)|}{\alpha}\right)dx \leq 1\right\}.
\end{equation*}

Associated to the space $L^{\Phi}(\Omega)$,  we can set  the Orlicz-Sobolev space $W^{1,\Phi}(\Omega)$  by 
$$ W^{1,\Phi}(\Omega) = \left\{u \in L^{\Phi}(\Omega)~: ~u_{x_i} \in L^{\Phi}(\Omega), ~i = 1, \cdots, N\right\} $$ and deduce that it is a Banach space  with respect to the norm
\begin{equation*}\label{norma2}
\|u\|_{W^{1,\Phi}} = \|u\|_\Phi + \|\nabla u\|_\Phi.
\end{equation*}
The Orlicz-Sobolev space $W_0^{1,\Phi}(\Omega)$ is naturally defined as the closure of $C_0^{\infty}(\Omega)$ in  $W^{1,\Phi}(\Omega)$-norm, under the hypothesis $(\phi_1)$.  For more information about the Orlicz and Orlicz-Sobolev spaces, we refer \cite{Adams}, \cite{MR0126722} and \cite{MR0482102}.

About $M$, let us assume
\begin{itemize}
	\item[$(M)$:] $M(t) \geq m_0 t^{\alpha -1}$ for all $t \geq 0$ and for some $\alpha > 0$ such that $\Phi_\alpha\prec\prec \Phi_*$, that is, $\displaystyle\lim_{t \to \infty}\frac{\Phi_\alpha(\tau t)}{\Phi_*(t)} = 0$ for all $\tau > 0$, where $\Phi_\alpha(t):=\Phi(t^\alpha)$. 
\end{itemize}

To conclude our assumptions, let us suppose that $ f:\Omega\times (0,\infty)\longrightarrow{\r^+}$ is a measurable function such that $f(x,t)=0$  a.e. in $\Omega\times (-\infty,0]$ and
	\begin{itemize}
	\item[$(f_0)$:]  $f(x,\cdot)\in\mathbf{C}\left({\r}-\{\tilde{a}\}\right) ~\mbox{for some}~\tilde{a}>0,
	~
	-\infty <f(x, \tilde{a}-0) < f(x, \tilde{a}+0) < \infty,~x \in \Omega,$
	
\noindent where	
	$$
	\displaystyle  f(x, \tilde{a}-0) :=\lim_{s  \to \tilde{a}^-} f(x, s ) ,~~    f(x, \tilde{a}+0) :=  \lim_{s  \to \tilde{a}^+} f(x, s ),
	$$
	
	\item[$(f_1)$:] there exists an odd increasing
	homeomorphism $h$ from $\mathbb{R}$ onto $\mathbb{R}$ and nonnegative constants $a_1,~a_2$ and $a_3$ such that
	$$ |\eta|\leq a_1 + a_2\widetilde H^{-1}\circ H(a_3|t|) ~\mbox{for all} ~ \eta\in\partial F(x,t), ~t \in \mathbb{R} ~\mbox{and}  ~ x \in \overline{\Omega}, $$
where 	$H(t) = \int_0^{|t|}h(s)ds$ is a N-function satisfying $\Delta_2$ ($\tilde{H}$ is the its complementary function) such that  $H\prec\prec\Phi_*$ and 
\begin{eqnarray}\label{aga}
	\frac{th(t)}{H(t)} \leq h_+ ~~\mbox{for all} ~t \geq t_0 ~\mbox{with } 1 < h_+ \leq \frac{\phi_-^*}{2} +1,
	\end{eqnarray}
for some $t_0 >0$,	
	{\item[$(f_2)$:] $\displaystyle\lim_{t \to 0^+} \frac{\sup_{\overline{\Omega}} F(x,t)}{t^{\alpha\phi_+}} = 0$, }
	\item[$(f_3)$:] $\displaystyle\lim_{t \to \infty} \frac{\sup_{\overline{\Omega}} F(x,t)}{t^{\alpha\phi_-}} = 0$.
\end{itemize}

Before stating the main results, let us clarify what we mean by a solution of  $(Q_{\lambda, \mu})$.

\begin{definition}\label{D41}
	A function $u \in W_0^{1,\Phi}(\Omega)$ is a solution to the problem $(Q_{\lambda,\mu})$ if $u > 0$ a.e in $\Omega$, $bu^{-\delta}\varphi \in L^1(\Omega)$ and 
\begin{equation*}\label{088}
	  M\left(\displaystyle\int_\Omega \Phi(|\nabla u|)dx \right)\displaystyle\int_\Omega a(|\nabla u|)\nabla u \nabla \varphi dx = \displaystyle\int_\Omega \left[\mu \frac{b(x)}{u^{\delta}} + \lambda f(x,u)\right]\varphi dx~\mbox{
		for all} ~\varphi \in W_0^{1,\Phi}(\Omega).
	\end{equation*}
\end{definition}

Under the hypothesis $(f_0)$, a solution $0<u \in W_0^{1,\Phi}(\Omega)$ of the problem $(Q_{\lambda,\mu})$ has to satisfy 
	\begin{equation}\label{zero}\int_{\Omega} \big( f(x,u(x)-0)- f(x,u(x)+0)\big)u \chi_{\{x \in \Omega ~:~ u(x)=\tilde{a}\}}(x)\varphi dx=0~\mbox{for all } \varphi \in  W_0^{1,\Phi}(\Omega),\end{equation}
	where $\chi_{\{x \in \Omega ~:~ u(x)=\tilde{a}\}}$ stands for the characteristic function of the set $\{x \in \Omega ~:~ u(x)=\tilde{a}\}$. Next, we state that  (\ref{zero}) is satisfied, under additional  assumptions on $f$ and $b$, by showing that $meas\{x \in \Omega ~:~ u(x)=\tilde{a}\}=0$, where $meas$ stands for the Lebesgue measure.

\begin{theorem}\label{med-u=a}
Assume $f$ satisfies  $(f_0)$, $(f_1)$ and $0<b \in L^1(\Omega)$ holds.  If $u\in W^{1,\Phi}_0(\Omega)$ is such that:  
\begin{enumerate}
\item[$i)$]   $u$ is either a local minimum or a local maximum  of $I$, then $meas\{x \in \Omega ~:~ u(x)=\tilde{a}\}=0$,
\item[$ii)$] $u$ is a critical point of $I$ and $b \in L^2_{\mathrm{loc}}(\Omega)$, then $meas\{x \in \Omega ~:~ \vert \nabla u(x) \vert=0\}=0$. In particular,  $meas\{x \in \Omega ~:~ u(x)=c\}=0$ for each $c>0$. 
\end{enumerate}
Moreover, if $u$  satisfies $i)$ or $ii)$ above, then:
\begin{enumerate}
\item[$(iii)$] $u$ is a solution of  Problem $(Q_{\lambda, \mu})$,
\item[$(iv)$] there exists $C>0$ such that $u(x) \geq Cd(x)$ for $x\in \overline{\Omega}$, where $d$ stands for the distance function to the boundary $\partial\Omega$,
\item[$(v)$] $u$ solves $(Q_{\lambda, \mu})$ almost everywhere in $\Omega$ if in addition $bd^{-\delta} \in L^{\tilde{H}}(\Omega)$.
\end{enumerate}
	\end{theorem} 

About multiplicity, our main result can be stated as follows.
\begin{theorem}\label{faraci} Assume $\delta >1$, $b \in L^1(\Omega)\cap L^2_{\mathrm{loc}}(\Omega)$, $(\phi_0)-(\phi_2),(f_0)  - (f_3)$ and $(M)$ hold.
	Then, the below claims are equivalents: 
	\begin{itemize}
		\item[$i)$] there exists $0 < u_0 \in W_0^{1,\Phi}(\Omega)$ such that $\displaystyle\int_{\Omega}bu_0^{1-\delta}dx < \infty$,
		\item[$ii)$] the problem $(S)$
		admits a $(unique)$ weak solution $u \in W_0^{1,\Phi}(\Omega)$ such that $u(x) \geq Cd(x)$ for $x\in \overline{\Omega}$ for some $C>0$ independent of $u$,
		\item[$iii)$] for each $\lambda > \lambda^*$,		
		 there exists $\mu_{\lambda} > 0$ such that for  $\mu \in (0, \mu_\lambda]$, the problem $(Q_{\lambda, \mu})$ admits at least three solutions, being two local minima and the other one a mountain pass critical point of the functional $I$, where
\begin{equation}
\label{08}
\!\!\!\!\lambda^* = \displaystyle\inf \left\{ \frac{\hat{M}\left(\displaystyle\int_\Omega\Phi(|\nabla u|)\right)}{\displaystyle\int_\Omega F(x,u)dx} ~: ~ u \in W_0^{1,\Phi}(\Omega)~ \mbox{and} ~\displaystyle\int_\Omega F(x,u)dx > 0 \right\}.
\end{equation}		
	\end{itemize}
Moreover, for each of such solutions the $meas\{x \in \Omega ~:~ u(x)=\tilde{a}\}=0$. Besides this, $u$ solves $(Q_{\lambda, \mu})$ almost everywhere in $\Omega$ if in addition $bd^{-\delta} \in L^{\tilde{H}}(\Omega)$ and if: 
\begin{enumerate}
\item[$iv)$] either $M$ is non-decreasing and $f(x,t)=f(x)$ for all $0<t<1$ and a.e. $x \in \Omega$,
\item[$v)$] or $M$ is such that a Comparison Principle holds  to Problem $(Q_{0,\mu})$ and $\alpha \phi_- >1$,
\end{enumerate}	
then there exists  $\tilde{a}^{\star}>0$ such that  $$meas\{x \in \Omega ~:~ u(x)>\tilde{a}~\mbox{and }u ~\mbox{is a solution of }(Q_{\lambda, \mu})\}>0$$
	for each $0<\tilde{a}<\tilde{a}^{\star}$ given.
\end{theorem}

\begin{remark} About the above theorem, we still highlight the following facts:
\begin{enumerate}
\item[$(i)$] the equivalency between $(i)$ and $(ii)$ holds true without assuming $b \in L^2_{\mathrm{loc}}(\Omega)$,
\item[$(ii)$] each one of such solutions given by $iii)$ is such that $u(x) \geq Cd(x)$ for $x\in \overline{\Omega}$, for some $C>0$ dependent on $u$.
\end{enumerate}
\end{remark}

In \cite{MR1037213}, {Lazer and Mckenna} has proven that problem $(S)$ admits solution still in $H_0^{1}(\Omega)$ if, and only if, $\delta < 3$ when $0 < b_0 \leq b \in L^{\infty}(\Omega)$ and $\Phi(t) = |t|^2/2 $ in $(S)$.  Mohammed, in \cite{MR2499900}, considered $\Phi(t) = |t|^p/p $ $(p > 1)$ in $(S)$ and proved that the sharp power in this case is given by $(2p-1)/(p-1). $ As a consequence of Theorem \ref{faraci}, we are able to find a $\delta_q> 1$ such that the problem $(S)$ still admits a solution in  $W_0^{1,\Phi}(\Omega)$ for all $  \delta <\delta_q $, where  $\delta_q$ depends on the summability $L^{q}(\Omega) $ of $b$. This is the content of the next corollary.

\begin{corollary}\label{est} Assume $(\phi_0), (\phi_1)$ and $(\phi_2)$ hold. If 
	$0<b \in L^q(\Omega)$ for some $q>1$ and 
	\begin{equation*}\label{otimo1}
	1 < \delta < \frac{q(2\phi_+ -1) -\phi_+}{q(\phi_+ -1)} := \delta_q,
	\end{equation*} 
	then the problem $(S)$
		admits (unique) weak solution.
\end{corollary}

  Although no answer about $\delta_q >1$ be the sharp power for the existence of solution still in $W_0^{1,\Phi}(\Omega)$ has been provided, we observe that $\delta_q \to ({2\phi_+ - 1})/({\phi_+ -1})$ as $q \to \infty$ and this limit is  the sharp value obtained both by \cite{MR1037213} and \cite{MR2499900} for the cases $\Phi(t) = |t|^2/2$ and $\Phi(t) = |t|^p/p ~(p > 1)$, respectively.
\medskip

In particular, as a consequence of Theorem \ref{faraci}  and Corollary \ref{est}, we have the following.

\begin{corollary}\label{carlos} Assume $(\phi_0), (\phi_1), (\phi_2), (M)$ and $(f_0)-(f_3)$ hold. If $b \in L^q(\Omega)$ for some $q>1$ and $1<\delta < \delta_q$, then for each $\lambda > \lambda^*$ given, there exists $\mu_{\lambda} > 0$ such that for $\mu \in (0, \mu_\lambda]$ the problem $(Q_{\lambda, \mu})$ admits at least three weak solutions with the same properties as those found in item$-iii)$ in Theorem $\ref{faraci}$. 
\end{corollary}

It is worth mentioning that the above theorems improve or complement the related results in the literature both by the presence of the Kirchhoff term, by the summability assumption on  the potential $b$, the strongly-singular term and  the non-homogeneity of the operator. Our results contribute to the literature principally by:
\begin{itemize}
\item[$i)$] Theorem \ref{med-u=a} unify some results on $\Delta_p$-Laplacian operator, with $1<p<\infty$, to $\Phi$-Laplacian operator, see for instance \cite{arcoya} and \cite{Lou}.
	\item[$ii)$]  Theorem \ref{faraci} establishes  necessary and sufficient conditions for existence of multiple solutions for the problem $(Q_{\lambda,\mu})$, by connecting and extending the principal result in Yijing \cite{MR3134198} to a non-homogeneous operator;
	\item[$iii)$]  Theorem \ref{faraci} extends the principal result in Faraci et.al \cite{MR3530211} and complements the main result in \cite{faraci}, principally by considering a non-homogeneous operator, an optimal condition on the pair  $(b,\delta)$ to existence of three solutions, a discontinuity of the Heaviside type and including a Kirchhoff term;
	\item[$iv)$] Corollary \ref{est} gives us an explicit  range of variation of $\delta$, in which the existence of  solution in $W_0^{1,\Phi}(\Omega)$ for $(S)$ is still guaranteed. In particular, when $\Phi(t) = |t|^p/p$ and $ b_0 \leq b(x) \in L^{\infty}(\Omega)$ for some constant $b_0 > 0$, the value $\delta_q$ coincides with the sharp values obtained in  \cite{MR1009077} and \cite{MR1037213};
	\item[$v)$] Corollary \ref{carlos} complements the principal result in \cite{MR3530211} by showing an explicit  variation to $\delta$, where the multiplicity is still ensured, namely, 
	$$0 < \delta <  \frac{p(N-1)}{N(p-1)} = \delta_{(p^*)'},$$
\end{itemize}

To ease the reading, from now on let us assume the assumptions ($\phi_0$), ($\phi_1$), ($\phi_2$), $(M)$ and gather below some functional that appear throughout the paper. 
\begin{itemize}
\item $\hat{M}(t) = \displaystyle\int_0^{t}M(s)ds$, $t\in \mathbb{R}$,
\item $\Psi_1(u) =  \hat{M}\left(\displaystyle\int_\Omega \Phi(|\nabla u|)dx \right) - \lambda \displaystyle\int_\Omega F(x,u)dx$,
\item $\Psi_2(u) = \displaystyle\int_\Omega G(x,u)dx$,
\item $\mathcal{P}(u)=\displaystyle\int_{\Omega}\Phi(|\nabla u|)dx$,
\item $J_1(u) := \left(\hat{M}\circ \mathcal{P}\right)(u)= \hat{M}\left( \displaystyle\int_{\Omega}\Phi(|\nabla u|)dx\right)$,
\item $J_2(u) = \displaystyle\int_\Omega F(x,u)dx,$ 
\item $I=\Psi_1 + \mu\Psi_2=J_1 - \lambda J_2 + \mu \Psi_2$,
\item $ -\left(M\circ \mathcal{P}\right)(\cdot)\Delta_\Phi (\cdot) : W^{1,\Phi}_0(\Omega) \to \left( W^{1,\Phi}_0(\Omega)  \right)^{\prime}$
is understood as 
$$\langle -\left(M\circ \mathcal{P}\right)(u)\Delta_\Phi u,\varphi\rangle:=\left(M\circ \mathcal{P}\right)(u)\int_{\Omega} a(|\nabla u|)\nabla u\nabla \varphi dx, ~ \forall ~\varphi\in W^{1,\Phi}_0(\Omega).$$

\end{itemize}

This paper is organized as follows. In Section 2, we present some preliminary knowledge on the Orlicz-Sobolev spaces and some results of non-smooth analysis related to our problem. The section 3 is reserved to prove Theorem \ref{med-u=a}, while in Section 4 we prove Theorem \ref{faraci}.

\section{Non-smooth analysis for locally Lipschitz functional}

In this section, we are going to remember some facts related to non-smooth analysis. However, one of the principal contribution of this section is establishing appropriated assumptions  under the N-function $\Phi$, the non-local term $M$ and the discontinuous function  $f$ that make  possible to approach $(ii\Longrightarrow iii)$, in Theorem \ref{faraci}, via Ricceri's Theorem \cite{MR2083908}. 

Under our hypotheses and the decomposition of the functional $I$  into $\Psi_1$ plus $\Psi_2$, that is,
\begin{equation}
\label{210}
I = \Psi_1 + \mu\Psi_2,
\end{equation}
we have written  $I$ as a sum of a locally Lipschitz functional $\Psi_1$  and a convex one $\Psi_2$ (see (\ref{t1}) and (\ref{t22})). Below, let us recall few notations and results on the  Critical Point Theory for the functional $\Psi_1$ and $\Psi_2$. We refer the reader to Carl, Le  \&  Motreanu \cite{carl},  Chang \cite{chang}, Clarke \cite{Clarke1} and references therein for more  details about this issue.

Let us begin by remembering that  the generalized directional derivative of $\Psi_1$ at $u\in W_0^{1,\Phi}(\Omega)$ in the direction of
$v\in W_0^{1,\Phi}(\Omega)$ is defined by
$$
\Psi_1^0(u;v)=\displaystyle
\limsup_{h\rightarrow0~\lambda\rightarrow0^+}\frac{\Psi_1(u+h+\lambda v)-\Psi_1(u+h)}{\lambda}
$$
 and the subdifferential of $\Psi_1^0(u;\cdot)$ at  $z\in W_0^{1,\Phi}(\Omega)$ is given by
$$
\partial \Psi_1^0(u;z)=\left\{\mu\in \left(W_0^{1,\Phi}(\Omega)\right)^{\prime}~:~\Psi_1^0(u;v)\geq \Psi_1^0(u;z)+ \langle\mu,v-z\rangle~\mbox{for all }v\in W_0^{1,\Phi}(\Omega)\right\},
$$
since $\Psi_1^0(u;\cdot)$ is a convex function.
In particular, $\partial \Psi_1^0(u;0)$ is named by the generalized gradient of $\Psi_1$ at $u$ and denoted by $\partial \Psi_1(u)$.

About  the functional $\Psi_2$, its effective domain is defined by $Dom(\Psi_2) = \{ u \in W_0^{1,\Phi}(\Omega)~: ~ \Psi_2(u) < \infty\}$ and a  point $u \in Dom(\Psi_2)$ is called a critical point of the functional $I$  if
	$$\Psi_1^0(u; v-u) + \Psi_2(v) - \Psi_2(u)	\geq 0, ~\forall ~v \in W_0^{1,\Phi}(\Omega) .$$
	
In this context, we say that $I$ satisfies the Palais-Smale condition (the condition (PS) for short) if:
\vspace{0.1cm}
\begin{center}
``$\{u_n\} \subset W_0^{1,\Phi}(\Omega)$ is such that $I(u_n) \to c $ and 
	$$ \Psi_1^0(u_n; v-u_n) + \Psi_2(v) - \Psi_2(u_n)	\geq -\epsilon_n\|v - u_n\|,  ~\forall ~v \in W_0^{1,\Phi}(\Omega),$$ where $\epsilon_n \to 0^+$, then $\{u_n\}$ possesses a convergent subsequence.''
\end{center}

In order to prove the next Lemma, let us define the functionals 
$$J_1(u) := \hat{M}(\mathcal{P}(u))~~\mbox{and}~~J_2(u) := \displaystyle\int_\Omega F(x,u)dx,$$ 
where $\mathcal{P}$ is defined by 
$$\mathcal{P}(u)=\int_{\Omega}\Phi(|\nabla u|)dx.$$

It is well know that, under the hypotheses  $(\phi_0)$ and $(\phi_1)$, the functional $\mathcal{P}$ is sequentially weakly lower semicontinuous and $C^1$  with 
$$\langle\mathcal{P}'(u), \varphi\rangle = \displaystyle\int_\Omega a(|\nabla u|)\nabla u\nabla \varphi dx, ~\forall  ~\varphi \in W_0^{1,\Phi}(\Omega).$$  Moreover, $\mathcal{P}': W_0^{1,\Phi}(\Omega) \to W_0^{-1,\tilde{\Phi}}(\Omega)$ is a 
strictly monotonic operator of the type $(S_+)$. Thus, we can rewrite $I$  as
\begin{equation}
\label{211}
I=  \Psi_1 + \mu \Psi_2 =J_1 - \lambda J_2 + \mu \Psi_2, 
\end{equation}
where $J_1$ is  $C^1$, $J_2$ is  locally Lipschitz and $\Psi_2$ is a convex functional.
\begin{lemma}\label{lllllll}  Suppose $(\phi_0)$, $(\phi_1)$, $(f_0)$ and $(f_1)$ holds. Then,
	\begin{itemize}
		\item[$i)$] $J_1 \in C^1(W_0^{1,\Phi}(\Omega), \mathbb{R}))$ and 
		$$ \langle J_1'(u), \varphi\rangle = M(\mathcal{P}(u))\displaystyle\int_\Omega a(|\nabla u|)\nabla u\nabla \varphi dx,~ \forall \varphi \in W_0^{1,\Phi}(\Omega),$$ 
		\item[$ii)$] $J_2 \in \mathrm{Lip}_{\mathrm{loc}}(W_0^{1,\Phi}(\Omega), \mathbb{R})$ and 
		$$\partial J_2(u)\subseteq \left\{w\in \left(L^{H}(\Omega)\right)^{\prime}~:~w(x)\in\partial F(x,u(x))~\mbox{a.e.}~x\in\Omega\right\}.$$	
		In particular, for each $w \in \partial J_2(u)$, there exists a unique $\omega \in L^{\tilde{H}}(\Omega)$ such that 
		$$\omega \in \left[f(x,u(x)-0), f(x,u(x)+0)\right]~\mbox{a.e.}~x\in\Omega~\mbox{and }\langle w,\varphi\rangle = \int_{\Omega} \omega \varphi dx,~\forall ~\varphi \in W_0^{1,\Phi}(\Omega),$$
		\item[$iii)$] $J_1'$ is of type $(S_+)$, that is, 
		$$  ``\mbox{if} ~ u_n \rightharpoonup u~ \mbox{and} ~\displaystyle\lim_{n \to \infty}\sup ~ \langle J_1'(u_n), u_n - u\rangle \leq 0, ~\mbox{then} ~u_n \to u ~\mbox{in} ~W_0^{1,\Phi}(\Omega)" . $$
		\item[$iv)$] if $u_n \rightharpoonup u$ in $W_0^{1,\Phi}(\Omega)$, then $$J_2^0(u_n;u_n-u)\to 0 ~\mbox{and}~ \langle \eta_n, u_n - u\rangle=\int_{\Omega}\eta_n(u_n-u)dx \to 0, ~\forall ~\eta_n\in \partial J_2(u_n),$$
		\item[$v)$] if $u_n \rightharpoonup u$ in $W_0^{1,\Phi}(\Omega)$, then $J_2(u_n) \to J_2(u)$,
		\item[$vi)$] $J_1$ is sequentially weakly lower semicontinuous in $W_0^{1,\Phi}(\Omega)$,
\item[$vii)$] $\Psi_1 \in \mathrm{Lip}_{\mathrm{loc}}(W_0^{1,\Phi}(\Omega); \mathbb{R})$ is sequentially weakly lower semicontinuous and $\Psi_1^0$ is of the type $(S_+)$.	
	\end{itemize}
\end{lemma}
\begin{proof}
	 First, we note that the item $i)$ is an immediate consequence of assumptions on $M$ and properties of $\mathcal{P}$. Next, we present a summary proof of the other items.
	\begin{itemize}
		\item[$ii)$] Let $\widetilde J_2:L^H(\Omega)\rightarrow \r$ be a functional defined by
		$\widetilde J_2 (u)=\int_{\Omega} F(x,u)dx,~u\in L^H(\Omega).$ So, it follows from Theorem 1.1 in  \cite{motreanu} that $\widetilde J_2\in \mathrm{Lip}_{\mathrm{loc}}(L^H(\Omega);\mathbb{R})$ and 
		$$
		\partial \widetilde J_2(u) \subseteq \left\{w\in \left(L^{H}(\Omega)\right)^{\prime}~:~w(x)\in\partial F(x,u(x))~\mbox{a.e.}~x\in\Omega\right\}.
		$$	
 	Since $\overline{W^{1,\Phi}_0(\Omega)}^{L^{H}}=L^{H}(\Omega)$, we are able to apply \cite[Theorem  2.2]{chang} to conclude that $J_2={\widetilde J_2}{{\big|}_{W_0^{1,\Phi}(\Omega)}}$ is locally Lipschitz continuous and
		$$
		\partial J_2(u) \subseteq \partial \widetilde J_2(u)\subseteq\left\{w\in \left(L^{H}(\Omega)\right)^{\prime}~:~w(x)\in\partial F(x,u(x))~\mbox{a.e.}~x\in\Omega\right\}.
		$$
		
		The conclusion of the proof is a direct consequence of Theorem 1.1 in \cite{motreanu} and classical  Riesz Theorem for Orlicz spaces, see for instance \cite{Rao}.
		\item[$iii)$] This conclusion is a consequence of item $i)$ and the fact that $\mathcal{P}'$ is of the type $(S_+)$.
		\item[$iv)$] Let $u_n \rightharpoonup u$ and  $\eta_n\in\partial J_2(u_n)$. Since $\eta_n\in \left(L^{H}(\Omega)\right)^{\prime}	$, the Riez Theorem for Orlicz spaces implies that there exists a unique $\eta_n \in  L^{\widetilde{H}}(\Omega)$, still denoted by $\eta_n$, such that
		$$\langle \eta_n,u_n - u\rangle=\int_{\Omega}\eta_n(u_n-u)dx.$$

Besides this, by using  $(f_1)$, $H\in \Delta_2$ and Young's inequality, we obtain 
		\begin{eqnarray}\label{conv-0}
			|\eta_n(u_n - u)| &\leq &a_1|u_n - u| + a_2 \widetilde H^{-1}\circ H(a_3|u_n - u|+a_3|u| )|u_n - u| \nonumber \\
			&\leq & C(|u_n - u| + H(|u_n - u|+|u|)),\nonumber 
		\end{eqnarray}
which leads us to conclude that $|\eta_n(u_n - u)| \leq g(x)$ for some $g \in L^1(\Omega)$, after using  the compact embedding $W_0^{1,\Phi}(\Omega) \hookrightarrow L^H(\Omega)$ and Lemma 5.3 in \cite{santos1}. As $u_n \to u$ a.e in $\Omega$, the first claim follows by Lebesgue Theorem.
				
To end the proof, it follows from Proposition 2.171 in \cite{carl}   that there exists $\widetilde \eta_n\in \partial J_2(u_n)$ such that 
		$J_2^0(u_n;v)=\langle \widetilde \eta_n,v\rangle$, for all $v\in W_0^{1,\Phi}(\Omega)$. Hence, we obtain from above conclusion that
		$J_2^0(u_n;u_n-u)=\langle \widetilde \eta_n,u_n-u\rangle\rightarrow 0.$
		
		\item[$v)$] As in the previous item, by using $(f_1)$ and dominated convergence the result follows. 
		\item[$vi)$] This item is a consequence of the continuity and monotonicity of $\hat{M}$ and the fact that $\mathcal{P}$ is sequentially weakly lower semicontinuous in $W_0^{1,\Phi}(\Omega)$.
		\item[$vii)$] By items $i)$ and $ii)$ above, we have  $\Psi_1\in \mathrm{Lip}_{\mathrm{loc}}(W^{1,\Phi}_0(\Omega);\r)$. Besides this,  we get from item $iv)$ and $(f_1)$ that $\Psi_1$ is sequentially weakly lower semicontinuous. Let $u_n \rightharpoonup u$ such that $\limsup_{n\rightarrow\infty}\Psi_1^0(u_n;u_n-u)\leq 0$. Then,  $(iii)$ and $(iv)$ above lead us to
	\begin{eqnarray}
	\limsup_{n\rightarrow\infty}\langle-\big(M\circ \mathcal{P}\big)(u_n)\Delta_\Phi u_n,u_n-u\rangle& = &\limsup_{n\rightarrow\infty}\langle
	-\big(M\circ \mathcal{P}\big)(u_n)\Delta_\Phi  u_n,u_n-u\rangle\nonumber\\
	&-&\lambda\lim_{n\rightarrow\infty}J_2^0(u_n;u_n-u)
		 = 	\limsup_{n\rightarrow\infty}\Psi_1^0(u_n;u_n-u)\leq 0, \nonumber
	\end{eqnarray}
which implies the claimed, after using  the $iii)$. This ends the proof.
		\fim
	\end{itemize}
\end{proof}

The next Lemma gives us some properties regarding $\Psi_2$.
\begin{lemma}\label{llllema}
	Assume $0 < b \in L^1(\Omega)$. If Problem $(S)$ 
	admits a solution in $W_0^{1,\Phi}(\Omega)$, then $\Psi_2$ is a proper functional. Besides this, $\Psi_2$ is convex, sequentially weakly lower semicontinuous and $\Psi_2(u)\neq -\infty$ for all $0<u \in W_0^{1,\Phi}(\Omega) $.
\end{lemma}
\begin{proof} First, note that $0 \leq G(x,u) \leq +\infty$ in $\Omega$ for all $ u \in W_0^{1,\Phi}(\Omega)$, so $\Psi_2(u) \neq -\infty$. Moreover, if $u_0 \in W_0^{1,\Phi}(\Omega)$ is a solution of $(S)$, then $u_0 \in Dom(\Psi_2)$, which proves $Dom(\Psi_2) \neq \emptyset$. 
	
The convexity follows directly from the definition of $\Psi_2$.
	Finally, by the Fatou's lemma, we conclude that $\Psi_2$ is sequentially weakly lower semicontinuous.
 \fim
\end{proof}

\begin{lemma}\label{I-coercive}
	Suppose $(\phi_0) - (\phi_3)$, $(M)$, $(f_1)$ and $(f_3)$ hold. Then, $I$ is a coercive functional.
\end{lemma}
\begin{proof}
First, by the assumption $(M)$ and Lemma 5.1 in \cite{santos1}, we have
\begin{equation}
\label{14}
\hat{M}\Big(\mathcal{P}(u)\Big) \geq \frac{m_0}{\alpha}\|\nabla u\|_{\Phi}^{\alpha\phi_-}~\mbox{for all} ~u \in W_0^{1,\Phi}(\Omega) ~\mbox{with}~\|\nabla u\|_\Phi \geq 1.
\end{equation}
Moreover, by taking $\epsilon>0$ small enough, it follows from $(f_1)$ and $(f_3)$ that $F(x,t) \leq C_1 + \epsilon |t|^{\alpha\phi_-}$  for all $x \in \Omega$, $t \in \mathbb{R}$ and for some $C_1 > 0$. Thus, by the embedding  $W_0^{1,\Phi}(\Omega)\hookrightarrow L^{\alpha\phi_-}(\Omega)$, which follows from the hypothesis $(\phi_3)$, we conclude
	\begin{equation}\label{coer}
	\Psi_1(u) \geq C_3\Big(\|\nabla u\|_{\Phi}^{\alpha\phi_-}  -1\Big)~\mbox{for all }~u \in W_0^{1,\Phi}(\Omega) ~\mbox{with}~\|\nabla u\|_\Phi \geq 1
	\end{equation} 
 for some $C_3 > 0$.	
Since $\delta > 1$, we have $\Psi_2(u) \geq 0$. Thus, after all these information and (\ref{210}), we conclude $I(u) \to \infty$ as $\|\nabla u\|_{\Phi} \to \infty,
$ that is, $I$ is coercive. This ends the proof.\fim
\end{proof}

\begin{lemma}\label{PS}
	 Suppose $(S)$ admits a solution in $W_0^{1,\Phi}(\Omega)$ and the assumptions $(\phi_0) - (\phi_3)$, $(M)$,$(f_1)$, $(f_3)$ hold. Then $I$ satisfies the $(PS)$ condition. 	
\end{lemma}
\begin{proof}
	Let $(u_n) \subset W_0^{1,\Phi}(\Omega)$ and $(\epsilon_n) \subset (0,\infty)$ be sequences such that $I(u_n) \to c \in \mathbb{R}$, $\epsilon_n \to 0$ and
	\begin{equation}\label{PS2}
	\Psi^0_1(u_n; \varphi - u_n) + \mu\Big(\Psi_2(\varphi) - \Psi_2(u_n)\Big) \geq -\epsilon_n \|\nabla(\varphi-u_n )\|_{\Phi} ~\mbox{for all} ~\varphi\in W_0^{1,\Phi}(\Omega) ~\mbox{and} ~ n \in \mathbb{N}.
	\end{equation} 
	
It follows from the coercivity of $I$, obtained in the previous Lemma, that $(u_n)$ is bounded in $W_0^{1,\Phi}(\Omega)$. Thus,  passing to a subsequence if necessary, we may assume that $u_n \rightharpoonup u$. So, by Lemmas \ref{lllllll}-$vii)$ and \ref{llllema}, we obtain that $I$ is sequentially weakly lower semicontinuous, which yields
	$$I(u) \leq \displaystyle\liminf_{n \to \infty} I(u_n) = c < \infty,$$ 
whence $\Psi_2(u) < \infty$. So, by  taking $\varphi = u$ in (\ref{PS2}),  we obtain 
	$$ -(-\Psi_1)^0(u_n; u_n - u) \leq \mu\Big(\Psi_2(u) - \Psi_2(u_n)\Big) + \epsilon_n\|\nabla(u_n- u)\|_\Phi ~\mbox{for} ~n \in \mathbb{N}.$$ 
	
Therefore, by using the previous inequality and the lower semicontinuity of $\Psi_2$, we get
	\begin{equation*}\label{limsup}
	\liminf_{n \to \infty}  (-\Psi_1)^0(u_n; u_n - u)  \geq 0,
	\end{equation*}
which leads to
		\begin{eqnarray*}
		0\leq \liminf_{n \to \infty}  (-\Psi_1)^0(u_n; u_n - u) & \leq & \liminf_{n \to \infty} \left[(-J_1)^0 (u_n; u_n - u)+\lambda J_2^0(u_n; u_n - u) \right]\\
		&=& \liminf_{n \to \infty} \langle-J_1'(u_n);u_n-u\rangle+\lambda\lim_{n\to \infty} J_2^0(u_n; u_n - u)\\
		&=&  -\limsup_{n \to \infty} \langle J_1'(u_n);u_n-u\rangle,
		\end{eqnarray*} 
after applying Lemma \ref{lllllll}-$iv)$.	Thus,  Lemma \ref{lllllll}-$iii)$ implies that $u_n \to u$ in $W_0^{1,\Phi}(\Omega)$ to a subsequence that ends the proof of Lemma.\fim  
\end{proof}

\begin{proposition}\label{2min1}
	Assume  $(\phi_0) - (\phi_3)$, $(M)$, $(f_1)$ and $(f_3)$ hold. Then, any strict local minimum of the functional $\Psi_1 = J_1 - \lambda J_2$ in the strong topology of $W_0^{1,\Phi}(\Omega)$ is so in the weak topology.
	
\end{proposition}
\begin{proof}
	We just need verify that, under these assumptions, the conditions of Theorem C in \cite{MR2536320} are fulfilled. Since  $W_0^{1,\Phi}(\Omega)$ is a reflexive and separable space, $J_1$ and $J_2$ are sequentially weakly lower semicontinuous and the functional $\Psi_1$ is coercive (see (\ref{coer})),  we just need to check that $J_1 \in \mathcal{W}_{W_0^{1,\Phi}}$, that is, 	
\begin{center}
	``if $u_n \rightharpoonup u$ and $\displaystyle\lim_{n \to \infty}\inf  J_1(u_n) \leq J_1(u)$, then $u_n \to u$ up to a subsequence''
\end{center}	
	to conclude the proof of the proposition,
	
In this direction, let us assume $ u_n\rightharpoonup u$ and $\displaystyle\lim_{n \to \infty}\inf  J_1(u_n) \leq J_1(u)$.  Since $J_1$ is sequentially weakly lower
	semicontinuous, we have $\displaystyle\lim_{n \to \infty}J_1(u_n) = J_1(u) $ for some subsequence, still denoted by $(u_n)$. Thus,  from this fact, continuity and monotonicity of $\hat{M}$ in $\mathbb{R}^+$, we obtain  $\displaystyle\lim_{n \to \infty}\mathcal{P}(u_n) = \mathcal{P}(u)$. Therefore, by the hypothesis $(\phi_1)$  we can apply \cite[Theorem 2.4.11 and Lemma 2.4.17]{MR2790542} to conclude that  $u_n \to u$ in $W_0^{1,\Phi}(\Omega)$. This ends the proof. \fim
\end{proof}

Below, let us connect the existence of solution to problem $(S)$ with existence of two local minima to the functional $I$. 
{\begin{lemma}\label{2min} Suppose $(S)$ admits a $W_0^{1,\Phi}(\Omega)$-solution, $(\phi_0)-(\phi_3),  (M)$ and $(f_1) - (f_3)$  hold. Then, for each $\lambda > \lambda^*$	
		 there exists $\mu_{\lambda} > 0$ such that for $\mu \in (0, \mu_\lambda]$ the functional $I$ has two local minima.	
\end{lemma}}
	\nd\proof:
	Fix $\lambda > \lambda^*$, where $\lambda^*>0$ was defined at (\ref{08}). 
	Since $\Psi_1$ is lower semicontinuous and coercive (see  Lemma \ref{lllllll}-$vii)$ and (\ref{coer})), there exists a global minimum $u_0 \in W_0^{1,\Phi}(\Omega)$  of $\Psi_1$ in $W_0^{1,\Phi}(\Omega)$ and, in particular,   $\Psi_1(u_0) \leq \Psi_1(0) =0$. If  $\Psi_1(u_0)  =0$, we would have 
	$$J_1(u) - \lambda J_2(u) =\Psi_1(u) \geq \Psi_1(u_0)=0~\mbox{for all }  u \in W_0^{1,\Phi}(\Omega),$$
which would yield $\lambda^* \geq \lambda$, but this is impossible.	
	
	Let us denote by $C>0$ the best embedding constant of  $W_0^{1,\Phi}(\Omega) \hookrightarrow L^{\alpha\phi_+}(\Omega)$ and take $0<\epsilon < ({m_0 C^{\alpha\phi_+}})/{\lambda\alpha}$. Thus, it follows from the assumptions $(f_2)$ and $(f_3)$  that  $F(x,t) \leq \epsilon t^{\alpha\phi_+} $ for all $t \in (0,m) \cup (M, \infty)$ for some $m> 0$ small enough and $M> 0$  large enough. 
	
Besides this,  if  $\|\nabla u\|_{\Phi} < \epsilon'$, then we have
	$$m^{\alpha \phi_+}\Big|[m \leq u \leq M]\Big| \leq \Big(\displaystyle\int_{[m \leq u \leq M]}u^{\alpha\phi_+}dx\Big)^{1/\alpha\phi^+} \leq \|u\|_{\alpha\phi_+} \leq C\|\nabla u\|_\Phi \leq C\epsilon',$$
that is, $\Big|[m \leq u \leq M]\Big| \leq C\epsilon'/m^{\alpha \phi_+}. $  
	
So, it follows from the above information and 	assumption  $(f_1)$ that
	\begin{eqnarray*}
		\displaystyle\int_\Omega F(x,u)dx & =& \displaystyle\int_{[u < m]} F(x,u)dx + \displaystyle\int_{[u >M]} F(x,u)dx + \displaystyle\int_{[m \leq u \leq M]} F(x,u)dx \\
		& \leq & \epsilon\displaystyle\int_{\Omega \backslash [m \leq u \leq M]} u^{\alpha\phi_+} dx + \displaystyle\sup_{m \leq t \leq M}F(x,t)\frac{C\epsilon'}{m^{\alpha\phi_+}} \leq \epsilon\displaystyle\int_{\Omega } u^{\alpha\phi_+} dx
	\end{eqnarray*} 
for some $\epsilon' > 0$ small  enough, which shows  $J_2(u) \leq \epsilon\|u\|_{\alpha\phi_+}^{\alpha\phi^+}$ for all $u \in W_0^{1,\Phi}(\Omega)$ with $\|\nabla u\|_\Phi \leq \epsilon'$.	

Therefore, we obtain from this fact, hypothesis $(M)$ and Lemma 5.1 in (\ref{14}) that
	\begin{eqnarray*}
		\Psi_1(u) 
		&\geq&  \frac{m_0}{\alpha}\|\nabla u\|_{\Phi}^{\alpha\phi_+}-\lambda\epsilon\|u\|_{\alpha\phi_+}^{\alpha\phi_+}   
		\geq  \frac{m_0 C^{\alpha\phi_+}}{\alpha}\|u\|_{\alpha\phi_+}^{\alpha\phi_+}-\lambda\epsilon\|u\|_{\alpha\phi_+}^{\alpha\phi_+} > 0 = \Psi_1(0)
	\end{eqnarray*} 
holds, whenever $\|\nabla u\|_\Phi < \epsilon'$ with $\epsilon' > 0$  such above, that is, $0$ is a strict local minimum of $\Psi_1$ in the strong topology. Hence, we obtain from Proposition \ref{2min1}  that $0$ is a local strict minimum of $\Psi_1$ in the weak topology as well, i.e, 
	there exists a weak neighborhood $V_w$ of $0$ such that 
	$$0 = \Psi_1(0) < \Psi_1(u) ~~\mbox{for all} ~u \in V_w \setminus \{0\}.$$
	
After these information and the assumption that the problem $(S)$ admits a solution in $W_0^{1,\Phi}(\Omega)$, we are able to follow the same strategy of the proof of Theorem 1.1 in \cite{MR3530211} to build disjoint open sets  $D_1$ and $D_2$, in the  strong topology, such that $0\in D_1$, $u_0\in D_2$ and to find $\tilde{\omega}_i \in D_i$
 such that  $\tilde{\omega}_1$ and $\tilde{\omega}_2$ are distinct local minima of $I$. This ends the proof.\fim
 \medskip

By applying Corollary 2.1 of \cite{Marano2002} for functional of the type locally Lipschiz plus convex (it is a version of  Corollary 3.3 in \cite{MR837231} that considers functional of the type $C^1$ plus convex),  Lemma \ref{PS} and Lemma \ref{2min}, we have.
\begin{corollary}\label{3pc}
	Suppose $(\phi_0)-(\phi_3), (b), (M)$ and $(f_1) - (f_3)$  hold. In addition, assume  that Problem $(S)$ admits a $W_0^{1,\Phi}(\Omega)$-solution. Then, for each $\lambda > \lambda^*$	
		 there exists $\mu_{\lambda} > 0$ such that for $\mu \in (0, \mu_\lambda]$ the functional $I$ has three critical points, being two of them local minima and the other one a mountain pass point to the functional $I$.
\end{corollary}

\section{Proof of Theorem \ref{med-u=a}}

Before starting the proof of Theorem \ref{med-u=a}, let us prove the two below Lemmas. 

\begin{lemma}[Multivalued solutions]\label{med-u=a1}
Assume $(\phi_0)-(\phi_3)$,  $(M)$, $(f_0)$,  $(f_1)$, $0<b \in L^1(\Omega)$ and $u \in W_0^{1,\Phi}(\Omega)$ be a critical point of $I$. Then:
\begin{enumerate}
\item[$(i)$] $u>0$ a.e. in $\Omega$ and there exist a $\eta \in \partial \Psi_2(u)$ and a $ \rho \in \left[f(x,u(x)-0, f(x,u(x)+0)\right] \subset  L^{\widetilde H}(\Omega)$  such that 
\begin{equation}
\label{2001}
\big(M\circ \mathcal{P}\big)(u)\int_{\Omega} a(|\nabla u|)\nabla u\nabla \varphi dx=\mu\langle \eta, \varphi \rangle  + \lambda \int_{\Omega} \rho \varphi dx~\mbox{for all } \varphi\in W_0^{1,\Phi}(\Omega),
\end{equation}
where $ \partial\Psi_2(u)$ stands for the subdifferential  of the convex functional $\Psi_2$ at $u$,
\item[$(ii)$] $bu^{-\delta}\varphi \in L^1(\Omega)$ for any $\varphi \in W_0^{1,\Phi}(\Omega)$. Besides this, 
$$\partial \Psi_2(u)=\{\eta\}~\mbox{and}~\langle \eta, \varphi\rangle =-\int_{\Omega} bu^{-\delta}\varphi dx ~\mbox{for all }\varphi \in W_0^{1,\Phi}(\Omega).$$
In particular, the equation $(\ref{2001})$ turns into
\begin{equation}
\label{14a}
\big(M\circ \mathcal{P}\big)(u)\int_{\Omega} a(|\nabla u|)\nabla u\nabla \varphi dx=  \int_{\Omega} \big[\mu bu^{-\delta} +\rho \big] \varphi  dx~\mbox{for all } \varphi\in W_0^{1,\Phi}(\Omega),
\end{equation}
\item[$(iii)$] there exists a $C>0$, dependent on $u$, such that $u(x) \geq Cd(x)$ for $x \in \Omega$, 
\item[$(iv)$] $ \rho+ bu^{-\delta}  \in L^2_{\mathrm{loc}}(\Omega) $  if in addition $b \in L^2_{\mathrm{loc}}(\Omega)$.
\end{enumerate}
	\end{lemma} 

\noindent\begin{proof} $of$ $(i)$. Since $u$ is a critical point of $I$ (see (\ref{211})), in particular, we have $u \in Dom(\Psi_2)$, which implies $\int_\Omega |G(x,u)|dx < \infty$, that is, $G(\cdot, u(\cdot))$ is finite a.e. in $\Omega$. Therefore, by the definition of $G$, we have $u > 0$ a.e in $\Omega$. 

Again, by $u\in W^{1,\Phi}_0(\Omega)$ be a critical point of $I$, it follows from \cite[Proposition 2.183]{carl}, that
\begin{equation*}
\label{18a}
0\in -\big(M\circ \mathcal{P}\big)(u)\Delta_\Phi u-\lambda\partial J_2(u)+\mu\partial \Psi_2(u),
\end{equation*}
	where $\partial J_2(u)$ stands for the generalized gradient  of the locally Lipschiz continuous functional $J_2$ at $u$. Thus,   there exist $\rho\in\partial J_2(u)$ and $\eta\in\partial \Psi_2(u)$ such that
\begin{equation}
\label{15a}
\langle -\big(M\circ \mathcal{P}\big)(u)\Delta_\Phi u,v\rangle=\lambda \langle\rho, v\rangle - \mu \langle\eta, v\rangle~\mbox{for all }v\in W^{1,\Phi}_0(\Omega).
\end{equation}	
	
So, it follows from Lemma \ref{lllllll}-$(ii)$ that there exists a unique $\rho \in L^{\tilde{H}}(\Omega)$, with $\rho \in \left[f(x,u(x)-0), f(x,u(x)+0)\right]$, such that  the equality (\ref{2001}) holds true. This ends the proof of $i)$.

\noindent Let us prove $ii)$. By (\ref{2001}) and $\rho \geq 0$, we have 
$$
\big(M\circ \mathcal{P}\big)(u)\int_{\Omega} a(|\nabla u|)\nabla u\nabla \varphi dx\geq -\mu\langle \eta, \varphi \rangle  ~\mbox{for all } 0\leq\varphi\in W_0^{1,\Phi}(\Omega),
$$
which implies, by definition of $\eta \in \partial \Psi_2(u)$, that
$$
\big(M\circ \mathcal{P}\big)(u)\int_{\Omega} a(|\nabla u|)\nabla u\nabla \varphi dx\geq \mu\displaystyle\int_{\Omega}\Big[\frac{G(x, u) - G(x,u+ t\varphi)}{t}\Big]dx   ~\mbox{for all } 0\leq\varphi\in W_0^{1,\Phi}(\Omega).
$$

Hence, by $u > 0$ a.e. in $\Omega$ and Fatou's lemma, we obtain 
\begin{eqnarray}
\label{16}
		\mu\displaystyle\int_\Omega b(x)u^{-\delta}\varphi dx &\leq &  \displaystyle\liminf_{t \to 0^+}\frac{1}{-\delta +1} \displaystyle\int_\Omega  {\mu}b(x) \left(\frac{\left(u + t\varphi\right)^{-\delta +1} - u^{-\delta +1}}{t}\right)dx \nonumber\\
		&\leq &  M\left(\mathcal{P}(u)\right)\displaystyle\int_\Omega a(|\nabla u|)\nabla u\nabla \varphi dx < \infty  ~\mbox{for all} ~0 \leq \varphi \in W_0^{1,\Phi}(\Omega), 
\end{eqnarray} 
that proves that $bu^{-\delta}\varphi \in L^1(\Omega)$ for any $\varphi \in W_0^{1,\Phi}(\Omega)$. To finish the proof of $ii)$, let  $\eta \in \partial \Psi_2(u)$. Then for $\epsilon \in (0,1)$, we have
	$$ \Psi_2\left(u - \epsilon u\right) - \Psi_2(u)
	\geq -\epsilon\langle\eta, u\rangle,  $$ which can be rewritten as 
	$$\frac{(1-\epsilon)^{-\delta +1} - 1}{(\delta -1)\epsilon}\displaystyle\int_\Omega b(x)u^{-\delta +1}dx \geq - \langle\eta, u\rangle. $$
	So, by doing $\epsilon \to 0^+$ in the previous inequality, we obtain
	\begin{equation}\label{desc}
	\displaystyle\int_\Omega b(x)u^{-\delta +1}dx \geq - \langle\eta, u\rangle.
	\end{equation}
	
	On the other hand, again by the fact that $\eta \in \partial\Psi_2(u)$, one has
	\begin{eqnarray*}\langle\eta, \varphi\rangle &\leq& \frac{\Psi_2\left(u + \epsilon \varphi\right) - \Psi_2(u)}{\epsilon} \\ &=& \frac{1}{-\delta +1}\displaystyle\int_\Omega b(x)\left(\frac{u^{-\delta+1} - (u+\epsilon \varphi)^{-\delta+1}}{\epsilon}\right)dx,\end{eqnarray*} 
 $\mbox{for all} ~0 \leq \varphi \in W_0^{1,\Phi}(\Omega)$ and $\epsilon>0$	given, which  yields
	\begin{eqnarray}\label{dvir}
	-\displaystyle\int_\Omega b(x)u^{-\delta} \varphi dx \geq \langle \eta , \varphi \rangle,
	\end{eqnarray}
after using Fatou's Lemma.

	By taking $\varphi = u$ in (\ref{dvir}) and combining this with (\ref{desc}), we obtain 
	\begin{equation}\label{equality}
	\langle \eta , u \rangle =  -\displaystyle\int_\Omega b(x)u^{-\delta+1}dx.
	\end{equation}
	
Besides this, by letting $\varphi \in W_0^{1,\Phi}(\Omega)$, testing (\ref{dvir}) with $(u+\epsilon\varphi)^+$ and using (\ref{equality}), we  get
	$$-\epsilon\displaystyle\int_\Omega b(x)u^{-\delta}\varphi dx \geq \epsilon\langle \eta, \varphi \rangle - \langle\eta, u\cdot \chi_{[u + \epsilon\varphi \leq 0]} \rangle - \epsilon \langle \eta, \varphi\cdot \chi_{[u + \epsilon\varphi \leq 0]}\rangle,$$
	which lead us to  
	\begin{equation}
	\label{20a}
	-\displaystyle\int_\Omega b(x)u^{-\delta}\varphi dx \geq \langle \eta, \varphi \rangle  -  \langle \eta, \varphi\cdot \chi_{[u + \epsilon\varphi \leq 0]}\rangle,
	\end{equation}
due (\ref{dvir}), that is,  $- \langle\eta, u\cdot \chi_{[u + \epsilon\varphi \leq 0]} \rangle \geq 0$.
	
	By using that $|[u + \epsilon\varphi \leq 0]| \to 0$ as $\epsilon \to 0$, the inequality (\ref{20a}) yields
	$$-\displaystyle\int_\Omega b(x)u^{-\delta}\varphi dx \geq \langle \eta, \varphi \rangle, ~\mbox{for all} ~\varphi \in W_0^{1,\Phi}(\Omega), $$ that is, 
	$$ \langle \eta, \varphi \rangle=-\displaystyle\int_\Omega b(x)u^{-\delta}\varphi dx , ~\mbox{for all} ~\varphi \in W_0^{1,\Phi}(\Omega). $$ 
This ends the proof of item $ii)$.

\noindent Now, we are ready to prove $iii)$. First, let us denote by $c_0 := M\left(\mathcal{P}(u)\right)>0$ and consider the problem
		\begin{equation}\label{sp}
		-\Delta_\Phi {v} =  \frac{\mu}{c_0}b_1(x)({v} +1)^{-\delta} ~~\mbox{in} ~\Omega, ~~~{v} = 0 ~~\mbox{on} ~\partial \Omega,
		\end{equation} where $b_1(x) = \min\{1, b(x)\}$. We know from Lemmas 4.2 and 5.1 in {\color{red}\cite{tmna}} that there exist a  unique solution of (\ref{sp}), say $\tilde{u}_1 \in W_0^{1,\Phi}(\Omega)$, and $C=C_u>0$ such that $\tilde{u}_1\geq Cd$ in $\Omega$. 
		
On the other hand,  we obtain from (\ref{16}) that
		$$\displaystyle\int_\Omega a(|\nabla u|)\nabla u\nabla \varphi dx \geq \displaystyle\int_\Omega  \frac{\mu}{c_0}b(x) u^{-\delta}\varphi dx \geq \displaystyle\int_\Omega  \frac{\mu}{c_0}b_1(x) (u + 1)^{-\delta}\varphi dx \mbox{ for all} ~0 \leq \varphi \in W_0^{1,\Phi}(\Omega),$$ that is, $u$ is a supersolution for the problem (\ref{sp}). Hence, 
		\begin{eqnarray*}
			0 &\leq &  \displaystyle\int_\Omega \left( a(|\nabla \tilde{u}_1|)\nabla \tilde{u}_1 - a(|\nabla u|)\nabla u \right)\nabla\left( \tilde{u}_1 - u \right)^+dx\\ & \leq & \frac{\mu}{c_0} \displaystyle\int_\Omega b_1(x)\left((\tilde{u}_1 + 1)^{-\delta} - (u + 1)^{-\delta}\right)\left( \tilde{u}_1 - u\right)^+dx \leq 0,\end{eqnarray*} which implies that $Cd\leq \tilde{u}_1 \leq u $ in $\Omega$ and this proves $(iii)$.
	
\noindent Let us prove $(iv)$. By  $(f_1)$ and property $\tilde{H}^{-1}\left(H(t) \right) \leq 2 \tilde{h}^{-1}(t) = 2 h(t)$  for all $t \in \mathbb{R}$ (the equality is due $h$ being continuous), we obtain 
\begin{align}
\label{c19}
|\rho| \leq a_2\widetilde{H}^{-1}\circ H(a_3(|u|+|\varphi|))+a_1 
\leq 2 a_2 h(a_3 \vert u \vert) + {a}_1 \leq C\left( 1 + u^{h_+ -1}\right), 
\end{align}
for some  $C>0$, where the last inequality is a consequence of $(\ref{aga})$ in $(f_1)$.  Hence, we obtain from (\ref{c19}),  $W_0^{1,\Phi}(\Omega)  \hookrightarrow L^{\phi^*_-}(\Omega)$ and $h_+ \leq {\phi_-^* }/{2} +1$  that $\rho \in L^2_{\mathrm{loc}}(\Omega)$. So, combining the fact that  $\rho \in L^2_{\mathrm{loc}}(\Omega)$ together with $(i)$ above,  the proof of $(iv)$ follows. This ends the proof of Lemma. \fim
\end{proof}	
	\medskip
	%Baseada em \cite[Lemma 2.1]{alvescarvalho}.
\begin{lemma}[Almost everywhere solutions]\label{phi-delta}
Assume $(\phi_0)-(\phi_3)$,  $(M)$, $(f_0)$,  $(f_1)$	and $bd^{-\delta} \in L^{\tilde{H}}(\Omega)$. Let $u\in W^{1,\Phi}_0(\Omega)$ be a critical point of $I$ and $\rho  \in   L^{\widetilde H}(\Omega)$ as in Lemma $\ref{med-u=a1}$. Then:
\begin{enumerate}
\item[$(i)$] $-\big(M\circ \mathcal{P}\big)(u)\Delta_\Phi u\in \left(L^{H}(\Omega)\right)^{\prime}$,
\item[$(ii)$]  there exists a unique representative of $-\big(M\circ \mathcal{P}\big)(u)\Delta_\Phi u$  in $ L^{\widetilde H}(\Omega)$, still denoted by $-\big(M\circ \mathcal{P}\big)(u)\Delta_\Phi u$, such that
\begin{equation}
\label{2005}
-\big(M\circ \mathcal{P}\big)(u)\Delta_\Phi u=\lambda \rho +\mu {b}{u^{-\delta}}~\mbox{a.e. in } \Omega.
\end{equation}
\end{enumerate}
\end{lemma}

\noindent	\proof $of$ $i)$ We have from (\ref{15a}) that
	$$\langle -\big(M\circ \mathcal{P}\big)(u)\Delta_\Phi u,\varphi\rangle=\lambda \langle\rho, \varphi\rangle - \mu \langle\eta, \varphi\rangle~\mbox{for all }\varphi\in W^{1,\Phi}_0(\Omega),$$
where $\eta\in\partial \Psi_2(u)\subset \left(W_0^{1,\Phi}(\Omega)\right)^{\prime}$ and $\rho\in\partial J_2(u) \subset \left(L^{H}(\Omega)\right)^{\prime}$ with this last inclusion due to the Lemma \ref{lllllll}-$(ii)$. Since  $bd^{-\delta} \in L^{\tilde{H}}(\Omega)$, we obtain from  Lemma \ref{med-u=a1}-$(ii)$ and $(iii)$ that $\eta \in  \left(L^{H}(\Omega)\right)^{\prime}$ as well.   Thus, we obtain from these information and $\overline{W_0^{1,\Phi}(\Omega)}^{\Vert \cdot \Vert_H}=L^H(\Omega)$ that $-\big(M\circ \mathcal{P}\big)(u)\Delta_\Phi u \in \left(L^{H}(\Omega)\right)^{\prime}$. This proves $i)$.
	
\noindent Let us prove $ii)$. It follows from item $i)$ and  Riesz Theorem for Orlicz spaces  that there exist a unique  element in $ L^{\widetilde H}(\Omega)$, still denoted by  $-\big(M\circ \mathcal{P}\big)(u)\Delta_\Phi u$, such that
\begin{eqnarray}\label{Riez-rep}
	\langle -M\circ \mathcal{P}(u)\Delta_\Phi u,\varphi\rangle=\int_\Omega \left(-M\circ \mathcal{P}(u)\Delta_\Phi u\right)\varphi dx~\mbox{for all }\varphi \in W^{1,\Phi}_0(\Omega),\nonumber
\end{eqnarray}
which implies by (\ref{14a}) that 
\begin{equation*}
\label{c22}
\int_\Omega\Big(-M\circ \mathcal{P}(u)\Delta_\Phi u-\lambda\rho -\mu bu^{-\delta}\Big)\varphi=0~\mbox{for all } \varphi\in W_0^{1,\Phi}(\Omega).
\end{equation*}
This ends the proof of Lemma. \fim
	\smallskip

\noindent {\bf Proof of Theorem \ref{med-u=a}-Conclusion.}  The proof of item $i)$ is inspired on ideas from \cite{ValdoJulio}, while for the proof of $ii)$ we borrow strategies from  \cite{Lou}. The item $(iii)$-$(v)$ are consequences of Lemmas \ref{med-u=a1} and \ref{phi-delta}.
\smallskip

	\noindent {\bf Proof of  $i)$:}  We just consider the case when $u$ is a local minimum for $I$. Similar arguments work when $u$ is a local maximum for $I$. In this case, it is readily  that
	\begin{align}\label{derivada}
	\lambda\int_\Omega\frac{F(x,u+\epsilon\varphi)-F(x,u)}{\epsilon}dx  -  \mu\int_\Omega \frac{G(x,u+\epsilon\varphi)-G(x,u)}{\epsilon}\varphi dx
	\leq \int_\Omega\frac{\hat M(\mathcal{P}(u+\epsilon\varphi))-\hat M(\mathcal{P}(u))}{\epsilon}dx
	\end{align}
holds for any $\varphi \in W_0^{1,\Phi}(\Omega)$ and any $\epsilon>0$ given.

Below, let us consider two cases. First, fix $0 \leq \varphi\in \mathbf{C}^\infty_0(\Omega)$.  So, we obtain from  Lebourg's theorem that  there exist  $t_0(x)\in(0,1)$ and 
$
	\xi_\epsilon\in \partial F(x,u+t_0\epsilon\varphi)
	$
such that
\begin{equation}
\label{20}
\frac{F(x,u+\epsilon\varphi)-F(x,u)}{\epsilon}=\xi_\epsilon\varphi,
\end{equation}		
for each $x \in \Omega$.

By using $(f_1)$, we are able to estimate $\xi_\epsilon$ by
	$$
	\begin{array}{lll}
	|\xi_\epsilon| \leq a_2\widetilde{H}^{-1}\circ H(a_3(|u|+|\varphi|))+a_1:=g,
	\end{array}	
	$$
where $g\in L^1(\Omega)$ is independent of $\epsilon>0$. Hence, coming back to (\ref{20}), we obtain
$$
	\left|\frac{F(x,u+\epsilon\varphi)-F(x,u)}{\epsilon}\right| \leq g\varphi\in L^1(\Omega)
	$$
	for every $\epsilon>0$ small enough.
	
Besides this, the right derivative of $F(x,\cdot)$ at $u $ is given by
	$$
	\lim_{\small{\epsilon\rightarrow 0^+}}\frac{F(x,u+\epsilon\varphi)-F(x,u)}{\epsilon}=f(x,u(x)+0) \varphi~\mbox{a.e.}~x\in\Omega,
	$$	
because $\varphi\geq 0$.
	
So, we are in position to apply Lebesgue's theorem, combined with Fatou's Lemma and Lemma \ref{phi-delta},  in (\ref{derivada}) to show that
\begin{equation*}
\label{extra}
\displaystyle\int_\Omega  \left(\lambda f(x,u(x)+0) +\mu\frac{b(x)}{u^\delta}\right)\varphi dx \leq  \displaystyle \big(M\circ \mathcal{P}\big)(u)\int_\Omega \phi(|\nabla u|)\nabla u\nabla \varphi dx 	=  \displaystyle \int_\Omega-\big(M\circ \mathcal{P}\big)(u)\Delta_\Phi u \varphi dx
\end{equation*}	
holds for any $0 \leq \varphi\in W^{1,\Phi}_0(\Omega)$, that is,
	$$
	-\big(M\circ \mathcal{P}\big)(u)\Delta_\Phi u(x)\geq \lambda f(x,u(x)+0) +\mu\frac{b(x)}{u^\delta}~\mbox{a.e.}~ x\in\Omega.
	$$
	
On the other hand, it follows from (\ref{2005}) and Lemma \ref{lllllll}-$(ii)$ that
$$-\big(M\circ \mathcal{P}\big)(u)\Delta_\Phi u(x)\leq \lambda f(x,u(x)+0) +\mu\frac{b(x)}{u^\delta}~\mbox{a.e.}~ x\in\Omega,$$
due to the fact that $u$ is a critical point	of $I$.

After these two inequalities, we obtain 
	\begin{equation}\label{e1}
	-\big(M\circ \mathcal{P}\big)(u)\Delta_\Phi u=\lambda f(x,u(x)+0) +\mu\frac{b(x)}{u^\delta}, ~\mbox{a.e.}~x\in\Omega.
	\end{equation}

Secondly, let us fix $\varphi\in \mathbf{C}^\infty_0(\Omega)$ with $\varphi \leq 0$.   By similar arguments as those done to prove the case $\varphi \geq 0$, we are able to show  that
	\begin{equation}\label{e2}
	-\big(M\circ \mathcal{P}\big)(u)\Delta_\Phi u=\lambda f(x,u(x)-0) +\mu\frac{b(x)}{u^\delta}~\mbox{a.e.}~x\in\Omega.
	\end{equation}
	holds. 
	
Finally, if  $meas\{x \in \Omega ~:~ u(x)=\tilde{a}\}>0$, then it would have from (\ref{e1})  and (\ref{e2}) that
	$$
	f(x,a-0)=f(x,\tilde{a}+0)~\mbox{a.e.}~x\in\{x \in \Omega ~:~ u(x)=\tilde{a}\},
	$$
but this is  impossible by $(f_0)$ so  $meas\{x \in \Omega ~:~ u(x)=\tilde{a}\}=0$. This ends the proof of $i)$.
\medskip
	
\noindent {\bf Proof of  $ii)$:} Since $u\in W^{1,\Phi}_0(\Omega)$ is a critical point of $I$, we obtain from Lemmas \ref{phi-delta}-$ii)$ that 
\begin{equation}
\label{201}
-\big(M\circ \mathcal{P}\big)(u)\Delta_\Phi u=\lambda \rho +\mu {b}{u^{-\delta}}:=h(x)~\mbox{a.e. in } \Omega.
\end{equation}
with $\rho \in [f(x,u-0),~~ f(x,u+0)]$. So, it follows from Lemma \ref{med-u=a1} and  \cite[Theorem 2.1]{Cianchi}, that 
\begin{equation*}
\label{c28}
a(|\nabla u|)|\nabla u|\in W^{1,2}_{loc}(\Omega).
\end{equation*}
	
Besides this, we have 	 
	$$
	\begin{array}{lll}
	\left|\displaystyle \nabla\left(\frac{a(|\nabla u|)|\nabla u|}{\epsilon+a(|\nabla u|)|\nabla u|}\right)\right| & = & \displaystyle \frac{\epsilon\big|\nabla(a(|\nabla u|)|\nabla u|)|}{[\epsilon+a(|\nabla u|)|\nabla u|]^2}  \leq  \displaystyle\frac{1}{\epsilon}|\nabla(a(|\nabla u|)|\nabla u|)|,
	\end{array}
	$$
which shows that 	
	$$\frac{a(|\nabla u|)|\nabla u|}{\epsilon+a(|\nabla u|)|\nabla u|}\in W^{1,2}_{loc}(\Omega)$$
	for each $\epsilon>0$ given and so 
	$$\frac{a(|\nabla u|)|\nabla u|}{\epsilon+a(|\nabla u|)|\nabla u|}\varphi\in W^{1,2}_{0}(\Omega)$$
	can be used as a test function for any $\epsilon>0$ and any $\varphi\in \mathbf{C}^\infty_0(\Omega)$ given.
	
By doing this, we get from (\ref{201})	 that
	\begin{equation}\label{int-reg-1}
	\begin{array}{lll}
	\displaystyle \int_\Omega h(x) \Big(\frac{a(|\nabla u|)|\nabla u|}{\epsilon+a(|\nabla u|)|\nabla u|}\varphi \Big)   & = & \displaystyle \big(M\circ \mathcal{P}\big)(u)\int_\Omega a(|\nabla u|)\nabla u \nabla\left(\frac{a(|\nabla u|)|\nabla u|}{\epsilon+a(|\nabla u|)|\nabla u|}\varphi\right)dx \\ \\
	& = & \displaystyle\big(M\circ \mathcal{P}\big)(u)\int_\Omega a(|\nabla u|) \frac{a(|\nabla u|)|\nabla u|}{\epsilon+a(|\nabla u|)|\nabla u|}\nabla u\nabla \varphi dx\\ \\
	& + & \displaystyle\big(M\circ \mathcal{P}\big)(u)\int_\Omega a(|\nabla u|)\varphi \frac{\epsilon}{[\epsilon+ a(|\nabla u|)|\nabla u|]^2}\nabla (a(|\nabla u|)|\nabla u|)\nabla u dx.
	\end{array}
	\end{equation}

Since,
$$	\begin{array}{lll}
\displaystyle\left|a(|\nabla u|)\varphi \frac{\epsilon}{[\epsilon+a(|\nabla u|)|\nabla u|]^2}\nabla (a(|\nabla u|)|\nabla u|)\nabla u\right|	& \leq &   \displaystyle|\varphi| \frac{\epsilon^2+(a(|\nabla u|)|\nabla u|)^2}{2[\epsilon+a(|\nabla u|)|\nabla u|]^2}|\nabla (a(|\nabla u|)|\nabla u|)| \\ \\
	& \leq & \displaystyle\frac{1}{2}|\varphi\nabla (a(|\nabla u|)|\nabla u|)|
	\end{array}
	$$
holds for any $\epsilon>0$, we are able to apply  Lebesgue Theorem  to the equalities in $(\ref{int-reg-1})$ to infer that
	\begin{eqnarray*}
	\displaystyle\int_{\Omega \setminus\left\{\nabla u\neq 0\right\}}h\varphi dx =  \displaystyle\big(M\circ \mathcal{P}\big)(u)\int_{\Omega \setminus \left\{\nabla u\neq 0\right\}} a(|\nabla u|)\nabla u \nabla \varphi 
	 =  \displaystyle\big(M\circ \mathcal{P}\big)(u)\int_{\Omega} a(|\nabla u|)\nabla u \nabla \varphi 
	 =  \displaystyle\int_{\Omega}h\varphi dx,
	\end{eqnarray*}
holds, which lead us to have $h(x)=0$ a.e. in $\{x\in \Omega~:~\nabla u=0\}$. As we already know from Lemma \ref{med-u=a1}-$(ii)$ that  $h(x)>0$ in $\Omega$, we obtain that 
	$meas \{x\in \Omega~:~\nabla u=0\}  =0$. So, it follows from 
Morey-Stampacchia's Theorem that $\{x\in \Omega~:~ u=c\}\subset \{x\in \Omega~:~\nabla u=0\}$ for any real constant $c$ given,
which shows that 	 $meas \{x\in \Omega~:~ u=c\}  =0.$ 

So, as a consequence of $(i)$ and $(ii)$ above, $\rho(x)=f(x,u(x))$ if $u(x)\neq  a$ and $\rho(x)\in [f(x, a-0),f(x, a+0)]$ if $u(x)=a$, we obtain that  $\rho(x) =f(x,u(x))$ a.e. in $\Omega$.  Finely, by applying Lemma \ref{med-u=a1}, we have $(iii)$ and $(iv)$, while Lemma \ref{phi-delta} implies $(v)$.
This ends the proof.\fim

\section{Proof of Theorem \ref{faraci}}

In this section, let us begin proving the equivalences among $(i)$, $(ii)$ and $(iii)$. To prove $(i\Longrightarrow ii)$, we borrow ideas from  Yijing \cite{MR3134198}, who treated this situation  in the context of homogeneous operators. The principal difficulty in doing this is to find appropriated assumptions under the N-function $\Phi$ to become possible to obtain compactness results for minimizing sequences on Nehari sets type, while the main obstacles to prove $(ii\Longrightarrow iii)$ were already got over in the last section. The $(iii\Longrightarrow i)$ is immediately. We will end this section ensuring that the discontinuity of the nonlinearity $f(x,\cdot)$ may be attained. 
\medskip

Let us begin by defining  the set 
	$$\mathcal{A}:= \Big\{ u \in W_0^{1,\Phi}(\Omega) ~: ~ \displaystyle\int_\Omega b(x)|u|^{1-\delta}dx < \infty\Big\}$$ and the subsets 
	$$\mathcal{N} :=\Big\{u \in W_0^{1, \Phi}(\Omega)  :  \displaystyle\int_\Omega \Big(a(|\nabla u|)|\nabla u|^2 -b(x)|u|^{1-\delta}\Big)dx \geq 0  \Big\}\subset \mathcal{A}$$
	and 
	$$ \mathcal{N}^* :=\Big\{u \in W_0^{1, \Phi}(\Omega) : \displaystyle\int_\Omega \Big(a(|\nabla u|)|\nabla u|^2 -b(x)|u|^{1-\delta}\Big)dx = 0  \Big\}.$$
\begin{lemma}
\label{l1}
Assume $(\phi_1)$ and $\mathcal{A} \neq \emptyset$. Then  $\mathcal{N}^*$ and $\mathcal{N}$ are non-empty sets and $\mathcal{N}$ is  unbounded set.
\end{lemma}	
\begin{proof} Take $u \in \mathcal{A}$.  So, it follows from $(\phi_1)$ and Lemma 5.1 in \cite{santos1}, that  
	\begin{equation}\label{abc}
	\displaystyle\int_\Omega \phi(t|\nabla u|)|\nabla u|dx \geq \frac{\phi_-}{t}\int_\Omega \Phi(t|\nabla u|)dx \geq \min\{t^{\phi_--1}, t^{\phi_+-1}\}\phi_-\displaystyle\int_\Omega \Phi(|\nabla u|)dx
	\end{equation} 
	and
	\begin{equation}\label{acd} \displaystyle\int_\Omega \phi(t|\nabla u|)|\nabla u|dx \leq \frac{\phi_+}{t}\int_\Omega \Phi(t|\nabla u|)dx \leq \max\{t^{\phi_--1}, t^{\phi_+-1}\}\phi_-\displaystyle\int_\Omega \Phi(|\nabla u|)dx \end{equation}
hold for $t>0 $ large enough.

So, we obtain from (\ref{abc}) and (\ref{acd}) that $\sigma'(t) \to \infty$ as $t \to \infty$ and $\sigma'(t) \to -\infty$ as $t \to 0^+$. Besides this, we have from $(\phi_1)$ again that $\sigma''(t) > 0$ for all $t> 0$, where $$\sigma(t) := J(tu) = \displaystyle\int_\Omega \Phi(t|\nabla u|)dx + \frac{t^{1-\delta}}{\delta -1}\displaystyle\int_\Omega b(x)|u|^{1-\delta}dx, ~t > 0$$
and so  there exists a  unique $t_* = t_*(u)$ (which is a global minimum of $\sigma$) such that $\sigma'(t_*) = 0$. This shows that $t_* u \in \mathcal{N}^*$. As another consequence of the above information, we have that $\sigma'(t) \geq 0$ for all $t > 0$ large enough, that is, $tu \in \mathcal{N}$ for all $t>0$ large enough. In particular, $\mathcal{N}$ is unbounded as well. This ends the proof.  \fim
\end{proof}

By using similar ideas as done Yijing \cite{MR3134198}  for the homogeneous case, we are able to prove the below Lemma in the context of non-local and non-homogeneous operator.
\begin{lemma}
\label{l2} Assume $(\phi_1)$ and $\mathcal{A} \neq \emptyset$. Then:
\begin{enumerate}
\item[$(i)$] the set $\mathcal{N}$ is strong closed,
\item[$(ii)$] $0$ is not an accumulation point of $\mathcal{N}$.
\end{enumerate}
\end{lemma}
	
To complete our basics tools to prove  Theorem \ref{faraci}, let us prove the below lemma that is interesting itself. 
	\begin{lemma}\label{l33}
		Assume that $0<b\in L^1(\Omega)$, $(\phi_0)$ and $(M)$ hold. Let $ g:\Omega\times(0,\infty)\longrightarrow \mathbb{R}$ be a Carath\'eodory function such that 
		\begin{equation}
		\label{31}
		\big(g(x,s)-g(x,t)\big)(s-t)\leq 0~\mbox{for all }s,t>0.
		\end{equation}
		Then the problem \begin{equation}\label{auxiliar}
		\left\{
		\begin{array}{ll}
		\displaystyle-M\left(\int_\Omega\Phi(|\nabla u|)dx\right)\Delta_\Phi u=\lambda\frac{b(x)}{u^\delta}+g(x,u),& \mbox{in}~\Omega\\
		u>0~\mbox{in}~\Omega,~u=0~\mbox{on}~\partial \Omega
		\end{array}
		\right.
		\end{equation}
		has at most one solution in $ W_0^{1,\Phi}(\Omega)$.
	\end{lemma}
	\begin{proof} First, we note that the fact of $M$ being non-increasing implies that $\hat{M}$ is convex. With similar arguments together with the hypotheses ($\phi_0$), we show that $\Phi$ convex as well. These facts and the hypotheses $(M)$ lead us to infer that the functional 
		$$J_1 (u):=\hat{M}\left(\int_{\Omega}\Phi(|\nabla u|)dx\right),~u \in W_0^{1,\Phi}(\Omega)$$ 
		is convex as well. 
		
		Let $u,v\in W_0^{1,\Phi}(\Omega)$ be two different  solutions of the problem	(\ref{auxiliar}). So, it follows from (\ref{31}) and the convexity of $J_1$, that
		\begin{eqnarray*}
		0\leq \langle J_1'(u)-J'_1(v),u-v\rangle &= &\lambda\int_\Omega \left( \frac{b}{u^\delta}-\frac{b}{v^\delta}\right)(u-v)dx+\int_{\Omega}(g(x,u)-g(x,v))(u-v)dx\nonumber\\
		&\leq& \lambda\int_\Omega \left( \frac{b}{u^\delta}-\frac{b}{v^\delta}\right)(u-v)dx<0,
		\end{eqnarray*}
		where the last inequality follows from $b,\delta>0$. This is impossible and so the proof of Lemma {\ref{l33}} is done. \fim
	\end{proof}	
\medskip	

\noindent{\bf Proof of Theorem \ref{faraci}-Conclusion.} We begin proving  the first implication.

\noindent{\bf Proof of $i)\Longrightarrow ii)$.} First, we note that the assumption $i)$ implies that $\mathcal{A} \neq \emptyset$. So, it follows from Lemmas \ref{l1} and \ref{l2} that $\mathcal{N} $ is a nonempty complete metric space. Moreover, by  Lemmas  \ref{lllllll} $(vi)$, Lemma \ref{llllema} and the fact that 
$$J(u) \geq \min\{\Vert \nabla u \Vert_{\Phi}^{\phi_-}, \Vert \nabla u \Vert_{\Phi}^{\phi_+}\}$$ 
we have that $J $ is lower semicontinuous and bounded below. Thus, by the Ekeland Variational Principle there exists a minimizing sequence $(u_n) \subset \mathcal{N}$ to $J$ constrained to  ${\mathcal{N}}$ such that:
	\begin{itemize}
		\item[$i)$] $J(u_n) \leq \displaystyle\inf_{\mathcal{N}} J + \frac{1}{n}; $
		\item[$ii)$] $J(u_n) \leq J(w) + \frac{1}{n}\|\nabla(u_n - w)\|_\Phi, ~\forall w \in \mathcal{N}. $
	\end{itemize}
Besides this, we may assume 	$u_n(x) > 0$ a.e in $\Omega$, because 
 $J(|u_n|) = J(u_n)$ and if we assume that  $u_n = 0$ in a measurable set $\Omega_0 \subset \Omega$, with $|\Omega_0| > 0$, then we would have from $u_n \in \mathcal{N}$, $b(x) > 0$ a.e in $\Omega$ and reverse H\"older inequality that 
	$$\infty > \phi_+\displaystyle\int_\Omega \Phi(|\nabla u_n|) dx \geq \displaystyle\int_{\Omega_0} b(x)u_n^{1-\delta} \geq 
	\Big(\displaystyle\int_{\Omega_0} b(x)^{1/\delta}dx \Big)^{\delta}\Big(\displaystyle\int_{\Omega_0} |u_n|dx\Big)^{1-\delta} = \infty, $$
	which is an absurd. Thus, $u_n(x) > 0$ a.e in $\Omega.$
	
	Since $J(u_n) \to \displaystyle\inf_{\mathcal{N}} J \geq 0,$ we have $$\min\{\|\nabla u_n\|_\Phi^{\phi_-}, \|\nabla u_n\|_\Phi^{\phi_+}\} \leq \displaystyle\int_\Omega \Phi(|\nabla u_n|)dx \leq \epsilon + \displaystyle\inf_{\mathcal{N}} J $$ for all $n$ large enough, which implies that $(u_n)$ is bounded. As a consequence of this, we have that
	$$\left\{\begin{array}{l}
	u_n \rightharpoonup u_* ~\mbox{in} ~W_0^{1,\Phi}(\Omega); \\
	u_n \to u_* ~\mbox{strongly in}~L^{G}(\Omega) ~\mbox{for all N-function } ~G \prec\prec \Phi_*; \\
	u_n \to u_* ~\mbox{a.e in} ~\Omega
	\end{array}\right.$$
	for some $u_* \in W_0^{1,\Phi}(\Omega)$.
	
By standard arguments, we are able to show that $J(u_*) = \displaystyle\inf_{\mathcal{N}} J$, that is,
	\begin{equation}\label{Aaa}
	\displaystyle\int_\Omega \Phi(|\nabla u_n|)dx + \frac{1}{\delta -1}\displaystyle\int_\Omega b(x)|u_n|^{1-\delta}dx \stackrel{n \to \infty}{\longrightarrow} \displaystyle\int_\Omega \Phi(|\nabla u_*|)dx + \frac{1}{\delta -1}\displaystyle\int_\Omega b(x)|u_*|^{1-\delta}dx
	\end{equation}
	holds. So, as a consequence of (\ref{Aaa}),  Fatou's Lemma and Lemma \ref{lllllll}$-vi)$, we obtain 
\begin{equation}
\label{21}
\displaystyle\lim_{n \to \infty}\displaystyle\int_\Omega \Phi(|\nabla u_n|)dx = \displaystyle\int_\Omega \Phi(|\nabla u_*|)dx.
\end{equation}

Thus, it follows from the assumption  $(\phi_1)$, Theorem 2.4.11 and Lemma 2.4.17 in \cite{MR2790542} that $W_0^{1,\Phi}(\Omega)$ is uniformly convex. This together with the weak convergence and (\ref{21}), lead us to conclude that 
	$
		u_n \to u_* ~\mbox{in} ~W_0^{1,\Phi}(\Omega). 
$	
After this strong convergence, we are able to follow similar arguments as done in Yijing \cite{MR3134198} in the homogeneous case to prove that 	
	\begin{equation*}
	\label{mesma}
	\displaystyle \int_\Omega a(|\nabla u_*|)\nabla u_*\nabla \varphi dx \geq \displaystyle\int_\Omega b(x)u_*^{-\delta}\varphi dx 
	\end{equation*}
holds for any $0 \leq \varphi \in W_0^{1,\Phi}(\Omega)$	 given. Hence, it follows from the same arguments as used to prove Lemma \ref{med-u=a1} that $u_*$ is a $W_0^{1,\Phi}(\Omega)$-solution of $(S)$ such that $u_* \geq Cd$ for some $C>0$ independent of $u$.
	\smallskip
	
\noindent{\bf Proof of $ii)\Longrightarrow iii)$.}  By  Corollary \ref{3pc}, there exist  three critical points to functional $I$, being two of them local minima  and the other one a mountain pass point to energy functional $I$. So, by  Theorem \ref{med-u=a} we know that each one of these critical point is a solution for the problem $(Q_{\lambda,\mu})$ that satisfy  the qualitative properties claimed.
\smallskip
	
\noindent{\bf Proof of $iii)\Longrightarrow i)$.}	Let $0 <u_0 \in W_0^{1,\Phi}(\Omega)$ be a solution of $(Q_{\lambda,\mu})$. Then  $u_0 \in Dom(\Psi_2)$, that is, 
	$\int_\Omega bu_0^{1-\delta}dx < \infty$. These ends the proof of the equivalences. 
\smallskip

Below, let us prove the items  $iv)$ and $(v)$. We are going to prove $iv)$ first. Let  $u=u_a$ be a solution of problem $(Q_{\lambda, \mu})$. Assume by contradiction that $u\leq a$ a.e. in $\Omega$ for any $a>0$. So, it follows from  $f(x,t)=f(x)$ for all $0<t<1$ and a.e. $x\in \Omega$ that $u_a \in  W_0^{1,\Phi}(\Omega)$
	is a solution of 
	$$
	\left\{
	\begin{array}{l}
	\displaystyle-M\left(\int_\Omega\Phi(|\nabla u|)dx\right)\Delta_\Phi u=\lambda\frac{b(x)}{u^\delta}+f(x)~ \mbox{in}~\Omega,\\
	u>0~\mbox{in}~\Omega,~u=0~\mbox{on}~\partial \Omega,
	\end{array}
	\right.
	$$
	that is, $u_a$ is constant in $a>0$ by  Lemma \ref{l33}. 
	
	On the other hand, by taking $\beta> \delta>1$, we have that $u_a^\beta>0$ can be used as a test function  in $(Q_{\lambda,\mu})$ and this  yields the inequality
	$$
	\begin{array}{lll}
	\displaystyle\beta M\left(\int_\Omega\Phi(|\nabla u_a|)dx\right)\int_\Omega a(|\nabla u_a|)|\nabla u_a| u_a^{\beta-1}dx &=&\displaystyle\int_{\Omega}b u_a^{\beta-\delta}dx +\int_{\Omega}f(x,u_a)u_a^\beta\\
	&\leq& \displaystyle|b|_1a^{\beta-\delta} +C|\Omega|(1+\widetilde H^{-1}\circ H(a))a^\beta
	\end{array}
	$$
for any $a>0$	given.

	So, by doing $a>0$ small enough we get an absurd, because the first term of the above inequality is a positive number that does not depends on $a>0$. This ends the proof of this item.
	\medskip
	
	Finally, we are going to prove $v)$. Let $u_a$ be a solution of problem $(Q_{\lambda,\mu})$. Assume by contradiction that $u\leq a$ a.e. in $\Omega$ for any $a>0$ again. So, it follows that $u_a $ is a super solution to problem 
	\begin{equation}
	\label{34}
	\left\{
	\begin{array}{l}
	\displaystyle-M\left(\int_\Omega\Phi(|\nabla u|)dx\right)\Delta_\Phi u=\lambda {b(x)} ~\mbox{in}~\Omega,\\
	u>0~\mbox{in}~\Omega,~u=0~\mbox{on}~\partial \Omega,
	\end{array}
	\right.
	\end{equation}
	whenever $a<1$.
	
	On the other hand,  we are able to show that the associated-energy functional to Problem (\ref{34}) is coercive due the assumption $\ell \alpha>1$. So, by following standard arguments, we show that there exists a non-trivial $0 \leq  v \in  W_0^{1,\Phi}(\Omega)$ solution for the problem  (\ref{34}). That is, we have
	$$
	\left\{
	\begin{array}{l}
	\displaystyle-M\left(\int_\Omega\Phi(|\nabla u_a|)dx\right)\Delta_\Phi u_a\geq \displaystyle-M\left(\int_\Omega\Phi(|\nabla v|)dx\right)\Delta_\Phi v~\mbox{in}~\Omega,\\
	u=v=0~\mbox{on}~\partial \Omega.
	\end{array}
	\right.
	$$
	
	So, it follows from the hypotheses that $M$ is such that a Comparison Principle holds, that $u_a\geq u>0$ for all $0<a\leq 1$. This fact together with the contradiction assumption lead us to  have $0 \leq u\leq u_a \leq a$ for all $0<a\leq 1$, which is impossible for $a>0$ small enough, because $u$ is non-trivial.  This ends the proof of item $v)$ and the proof of Theorem  \ref{faraci}.
	\fim
	\medskip
	
\noindent{\bf Proof of Corollary \ref{est}:}
	By the implication $(i \Longrightarrow ii)$ in Theorem \ref{faraci}, it suffices to exhibit a $u_0 \in W_0^{1,\Phi}(\Omega)$ such that $\displaystyle\int_\Omega bu_0^{1-\delta}dx < \infty. $ Let us construct a such one.  First, we note that the regularity of the domain $\Omega$ implies that there exists an $\epsilon > 0$ sufficiently small such that $d \in C^2(\overline{\Omega}_{2\epsilon})$ and $|\nabla d(x)| = 1 $ in $\Omega_{2\epsilon}$, where $d(x):= dist(x, \partial\Omega)$ and $\Omega_{2\epsilon} = \{x \in \Omega ~ : ~ d(x) < 2\epsilon\}$.
	With these, define
	$$u_0(x) = \left\{
	\begin{array}{l}
	d(x)^\theta ~~\mbox{if} ~d(x) < \epsilon, \\
	\epsilon^{\theta} + \displaystyle\int_{\epsilon}^{d(x)} \theta \epsilon^{\theta -1} \Big(\frac{2\epsilon -t}{\epsilon}\Big)^{2/(\phi_- -1)}dt ~~\mbox{if} ~\epsilon \leq d(x) < 2\epsilon, \\
	\epsilon^{\theta} + \displaystyle\int_{\epsilon}^{2\epsilon} \theta \epsilon^{\theta -1} \Big(\frac{2\epsilon -t}{\epsilon}\Big)^{2/(\phi_- -1)}dt ~~\mbox{if} ~\epsilon \leq d(x) < 2\epsilon
	\end{array}\right.$$
for each $\epsilon > 0$ given, where $0 < \theta < 1 $ will be chosen later. 
	
	A simple calculation yields
	$$\nabla u_0(x) =\left\{
	\begin{array}{l}
	\theta d(x)^{\theta -1}\nabla d(x) ~~\mbox{if} ~d(x) < \epsilon ,\\
	\theta \epsilon^{\theta -1} \Big(\frac{2\epsilon -d(x)}{\epsilon}\Big)^{2/(\phi_- -1)}\nabla d(x) ~~\mbox{if} ~\epsilon \leq d(x) < 2\epsilon, \\
	0 ~~\mbox{if} ~\epsilon \leq d(x) < 2\epsilon,
	\end{array}\right.$$
which implies that  $u_0 \in W_0^{1,\Phi}(\Omega)$ if
	\begin{equation}\label{lazer}\displaystyle\int_{\Omega_\epsilon} \Phi(\theta d(x)^{\theta -1}|\nabla d(x)|) dx < \infty. \end{equation}
	
Since $\vert \nabla u \vert = 1$ in ${\Omega_\epsilon} $, we obtain from Lemma 5.1 in (\ref{14}) that	 
	$$\displaystyle\int_{\Omega_\epsilon} \Phi(\theta d(x)^{\theta -1}|\nabla d(x)|) dx = \displaystyle\int_{\Omega_\epsilon} \Phi(\theta d(x)^{\theta -1}) dx \stackrel{\theta < 1}{\leq} C\displaystyle\int_{\Omega_\epsilon} d(x)^{(\theta -1)\phi_+} dx $$ 
that lead us to show (\ref{lazer}) for $\theta$ such that  $(\theta -1)\phi_+ > -1$, due well-known result in \cite{MR1037213}. That is,  for such $\theta$, we have that $u_0 \in W_0^{1,\Phi}(\Omega)$.
	
To complete the exhibition, if $0<\theta<1$ is such that  ${\theta q(1-\delta)} > 1 -q$, we  have
 \begin{equation*}\label{gn}
	\displaystyle\int_{\Omega_\epsilon} b(x)d(x)^{\theta(1-\delta)}dx \leq \Big(\int_{\Omega}b(x)^qdx  \Big)^{1/q}\Big(\displaystyle\int_{\Omega_\epsilon}d(x)^{\theta(1-\delta)q'}dx < \infty\Big)^{1/q'}<\infty,
	\end{equation*}
because  $b \in L^q(\Omega)$ and  the result in \cite{MR1037213} again.
	
Finally, to occur (\ref{lazer}) and (\ref{gn}) simultaneously, we have to be able to choose a $0<\theta<1$ satisfying at same time $(\theta -1)\phi_+ > -1$ and  ${\theta q(1-\delta)} > 1 -q$.  We can do these by controlling the range of $\delta$. Since 
$$1 -\frac{1}{\phi_+} < \frac{q -1}{q(\delta -1)}~\mbox{if, and only if,}~0 < \delta < \frac{q(2\phi_+ -1) -\phi_+}{q(\phi_+ -1)},$$ 
we are able to pick a  $$\theta \in \Big(1 -\frac{1}{\phi_+}, \min\Big\{1, \frac{q -1}{q(\delta -1)}\Big\}\Big) \subset (0,1),$$
whenever $\delta$ range as above. This proves that  $u_0$, defined as above, satisfies the condition of item $i)$ in Theorem \ref{faraci}. This  finishes the proof. \fim

\end{document}